\documentclass[11pt,a4paper,reqno]{amsart}
\usepackage[utf8]{inputenc}
\usepackage[english]{babel}
\usepackage{amsmath}
\usepackage{amsfonts}
\usepackage{amssymb}
\usepackage{mathrsfs}
\usepackage{latexsym}
\usepackage{natbib}
\usepackage[margin=3cm]{geometry}
\usepackage{xcolor}
\usepackage{color}
\usepackage{amsthm}
\usepackage{graphicx,color}
\usepackage{enumitem}

\usepackage[plainpages=false,colorlinks,hyperindex,bookmarksopen,linkcolor=black,citecolor=black,urlcolor=black]{hyperref}
\newcommand*\bigcdot{\mathpalette\bigcdot@{.3}}
\usepackage[capitalize]{cleveref}

\bibpunct{[}{]}{;}{n}{,}{,}

\newtheorem{tm}{Theorem}[section]
\newtheorem*{tm*}{Theorem}

\newtheorem{coro}[tm]{Corollary}
\newtheorem{lemma}[tm]{Lemma}
\newtheorem{prop}[tm]{Proposition}

\theoremstyle{remark}
\newtheorem{remark}[tm]{Remark}

\newcommand{\cG}{\mathcal{G}}
\newcommand{\cF}{\mathcal{F}}
\newcommand{\NN}{\mathbb{N}}
\newcommand{\RR}{\mathbb{R}}

\newcommand{\sfv}{\mathsf{v}}
\newcommand{\sfx}{\mathsf{x}}

\title{Sticky diffusions on star graphs : characterization and It\^o formula}
\author{Jules Berry}
\address[J. Berry]{Univ. Rennes, INSA, CNRS, IRMAR - UMR 6625, Rennes F-35000, France}
\email{jules.berry@insa-rennes.fr }

\author{Fausto Colantoni}
\address[F. Colantoni]{Department of Basic and Applied Sciences for Engineering, Sapienza University of Rome, Rome, Italy}
\email{fausto.colantoni@uniroma1.it }

\date{\today}
\begin{document}
\maketitle
\begin{abstract}
We investigate continuous diffusions on star graphs with sticky behavior at the vertex. These are Markov processes with continuous paths having a positive occupation time at the vertex. We characterize sticky diffusions as time-changed nonsticky diffusions by adapting the classical technique of It\^o and McKean. We prove a form of It\^o formula, also known as  Freidlin-Sheu formula, for this type of process. As an intermediate step, we also obtain a stochastic differential equation satisfied by the radial component of the process. These results generalize those already known for sticky diffusions on a half-line and skew sticky diffusions on the real line.
\end{abstract}
\section{Introduction}
In the unidimensional setting, sticky processes find their roots\footnote{See \cite{P2015} for historical aspects of these matters.} in the seminal work of Feller \cite{F1952,F1954,F1957}. These processes are characterized by a second order term in the boundary condition defining the domain of their generator. For instance, the sticky (reflected) Brownian motion on $\RR_+$ was first obtained by considering the Feller semigroup with generator given by
\begin{equation}\label{eq:sticky_BM}
\begin{split}
 L f(x) & = \frac{1}{2} f''(x) \quad \textnormal{for } x > 0, \\
 D(L)  & = \left \{ f \in C^2((0,+\infty)) \cap C_0(\RR_+) : \eta f''(0^+) = f'(0^+) \right \},
\end{split}
\end{equation}
where $\eta \geq 0$ is the stickiness parameter. The term ``sticky'' is motivated by the fact that, when $\eta > 0$, the occupation time of the process at $0$ is positive, despite the fact that the set of zeros of the process has empty interior, while it is null for the reflected Brownian motion (which corresponds to the case $\eta =0$). It was later established by It\^o and McKean in \cite{ItoMcKean} that the sticky Brownian motion could be characterized by suitable random time-changes of the reflected Brownian motion. More precisely, let $B$ be a reflected Brownian motion and denote by $\ell_B(t)$ its local time at $0$. Define the continuous increasing process $V \colon \RR_+ \to \RR_+$ by
\begin{equation}\label{eq:time_change_BM}
 V(t) := t + \eta \ell_B(t),
\end{equation}
consider $V^{-1} \colon \RR_+ \to \RR_+$ its right inverse, \textit{i.e.},
\[
 V^{-1}(t) = \inf \left \{ s > 0 : V(s) > t \right \},
\]
and define the time-changed process $\hat B(t) := B(V^{-1}(t))$. It can then be proved that $\hat B$ is a Markov process whose semigroup has generator given by \eqref{eq:sticky_BM}.
From the definition \eqref{eq:time_change_BM}, we observe that $V$ is a strictly increasing process that grows faster than $t$ when $B$ is at $0$. Therefore, its inverse $V^{-1}$ grows more slowly than $t$, causing the time-changed process $\hat B$ to become ``stuck'' at $0$. This is a direct consequence of the fact that the local time increases only at $0$, while it remains constant otherwise.
We also note that, when $\eta = 0$, we have $V(t) = t = V^{-1}(t)$, meaning that $\hat B$ coincides with $B$. When $\eta > 0$, outside $0$, $\hat B$ follows the trajectories of $B$, but with a delay.
Recently, it was proved in \cite{EP2014} that the sticky Brownian motion satisfies the stochastic differential equation
\begin{equation}
 \begin{cases}
  dX(t) = \mathbf 1_{\{ X(t) > 0 \}} dW(t) + \frac{1}{2} d\ell_X(t), \\
  \eta \ell_X(t) = \int_0^t \mathbf 1_{\{X(s) = 0\}} ds,
 \end{cases}
\end{equation}
for some one-dimensional Brownian motion $W$, where $\ell_X$ is the local time of $X$ at $0$.
More generally, one can consider continuous sticky diffusions on $\RR_+$ as the Markov processes with generator given by
\begin{equation}\label{eq:sticky_diffusion}
\begin{split}
 L f(x) & = \frac{1}{2} \sigma^2(x) f''(x) + b(x) f'(x) \quad \textnormal{for } x > 0, \\
 D(L)  & = \left \{ f \in C^2((0,+\infty)) \cap C_0(\RR_+) : \eta L f(0^+) = f'(0^+) \right \}.
\end{split}
\end{equation}
In \cite{salins}, it is established that such diffusions are random time-changes of the reflected diffusion ($\eta = 0$) and satisfy
\begin{equation}
 \begin{cases}
  dX(t) = \sigma(X(t)) \mathbf 1_{\{ X(t) > 0 \}} dW(t) + b(X(t))\mathbf 1_{\{ X(t) > 0 \}} dt + \frac{1}{2} d\ell_X(t), \\
  \eta \ell_X(t) = \int_0^t \mathbf 1_{\{X(s) = 0\}} ds,
 \end{cases}
\end{equation}
for some Brownian motion $W$, where $\ell_X$ is the local time of $X$ at $0$. \\

Another interesting process is the skew Brownian motion introduced in \cite{walsh1978diffusion}. Loosely speaking, it is Brownian motion on $\RR$ which, upon reaching $0$, is reflected in $\RR_+$ with probability $p$ and in $\RR_-$ with probability $1-p$. The generator of this process is given by
\begin{equation}\label{eq:skew_BM}
 \begin{split}
  L f(x) & = \frac{1}{2} f''(x) \quad \textnormal{for } x \neq 0, \\
 D(L)  & = \left \{ f \in f \in C^2(\RR \setminus \{ 0 \}) \cap C_0(\RR) : \,  \begin{array}{|c}  f''(0^+) = f''(0^-), \\  p f'(0^+) = (1 - p) f'(0^-) \end{array} \right \}.
 \end{split}
\end{equation}
We refer the reader to \cite{lejay} and references therein for further information on this class of processes. The version of this process exhibiting a sticky behavior at $0$ is then obtained by replacing $D(L)$ in \eqref{eq:skew_BM} by
\begin{equation}
 D(L)  = \left \{ f \in C^2(\RR \setminus \{0\}) \cap C_0(\RR) : \,  \begin{array}{|c} f''(0^+) = f''(0^-), \\  \eta f''(0) =  p f'(0^+) - (1 - p) f'(0^-) \end{array} \right \}.
\end{equation}
This process was recently studied in \cite{salins,skew-sticky} where it is shown that the process can be obtained as a random time-change of a skew Brownian motion and that it is a weak solution to some stochastic differential equation. In addition, in \cite{skew-sticky}, an explicit expression for the transition kernel of the process is given. Similar results were previously obtained in \cite{H2007,EP2014,B2014} in the special case $p = 1/2$, in which case it is also called two sided sticky Brownian motion. For more general diffusion processes, considered in \cite{salins}, the generator takes the form
\begin{equation}\label{eq:sticky_diffusion}
\begin{split}
 L f(x) & = \frac{1}{2} \sigma^2(x) f''(x) + b(x) f'(x) \quad \textnormal{for } x \neq 0, \\
 D(L)  & = \left \{ f \in C^2(\RR \setminus \{0\}) \cap C_0(\RR) : \begin{array}{|c} Lf(0^+) = Lf(0^-), \\  \eta L f(0) = p f'(0^+) - (1 - p) f'(0^-) \end{array} \right \}.
\end{split}
\end{equation}
In \cite{salins}, it is established that these processes can be constructed from the skew diffusions  processes ($\eta = 0$) by using the It\^o-McKean technique and satisfy
\begin{equation}
 \begin{cases}
  dX(t) = \sigma(X(t)) \mathbf 1_{\{ X(t) \neq 0 \}} dW(t) + b(X(t))\mathbf 1_{\{ X(t) \neq 0 \}} dt + (2p - 1) d\ell_X(t), \\
  \eta \ell_X(t) = \int_0^t \mathbf 1_{\{X(s) = 0\}} ds,
 \end{cases}
\end{equation}
for some Brownian motion $W$, where $\ell_X(t)$ is the symmetric local time of $X$ at $0$. \\

Stochastic processes on graphs are an intriguing extension of unidimensional ones, useful for many applications involving complex systems on networks. They were first introduced by Feller's abstract approach in \cite{FreidlinWentzell}. A fundamental example is certainly the Brownian motion on a star graph, whose paths are constructed in \cite{russians1,russians2}. We can think of it as a stochastic process defined on a number of edges, each of which is a copy of the positive half-line, connected by a single vertex. On each of these edges, the process behaves like a reflected Brownian motion and randomly switches to another one upon reaching the vertex. In \cite{russians1,russians2}, the authors provide a full characterization of the infinitesimal generator of the corresponding Feller semigroup. They also provide a construction of the paths of the process in every case by adapting techniques used for Brownian motions on an interval. In particular the Brownian motion with sticky behavior at the vertex is obtained by using the It\^o-McKean random time-change. The construction was then generalized to generic metric graphs, i.e. graphs with more than one vertex, in \cite{russians3}. Recently, the analysis and construction of Brownian motion with a sticky point, even under non-local dynamic conditions leading to trapping behavior at the vertex, has also been addressed in \cite{BonDov}. A more general case is the Walsh Brownian motion \cite{walsh1978diffusion,WalshBM}, where once the process reaches the vertex, it selects a new edge according to a continuous distribution on $ [0, 2\pi)$. An It\^o formula for this setup is derived in \cite{Hajri-Ito}.

Paths of general nondegenerate diffusions on star graphs were studied in \cite{FreidlinSheu} under the assumption that the process is nonsticky at the vertex. There, it is established that the radial component of the process is a weak solution to some stochastic differential equation and an It\^o formula is obtained. The authors also prove the existence of a process which can be interpreted as the local time of the nonsticky process at the vertex. Our goal is to extend these results to processes exhibiting sticky behavior at the vertex. In addition to \cite{FreidlinWentzell}, the semigroups associated to sticky diffusions on metric graphs have also been analysed in \cite{sticky-graphs} in the special case of pure diffusion processes and a form of singular convergence theorem is also provided. We also mention that, recently, a more general class of nonsticky diffusion processes has been proposed in \cite{MO2025}.\\

The main result of this work is the derivation of an It\^o formula for general nondegererate diffusions on star graphs. As an intermediate steps, we also obtain a stochastic differential equation satisfied by radial component of the process and a characterization of sticky diffusions as time-changes of nonsticky ones. On the one hand, our results extend those available in the unidimensional setting \cite{salins,skew-sticky} by allowing more than one or two edges. On the other hand, we generalize those from \cite{FreidlinSheu} by considering sticky processes. This is achieved by making a systematic use of the It\^o-McKean random time change in this setting, where we use the local time process obtained in \cite{FreidlinSheu} to define the time change. We also apply this It\^o formula to obtain representations formulas for some differential equations posed on the star graph. In particular we obtain a Feynman-Kac formula in this context.
Our choice to focus on diffusions on star graphs was made in order to simplify notations and to avoid obfuscating the exposition with technicalities. However, we expect that our conclusions can be extended to general metric graphs.
\\

As was already mentioned, our main motivation in writing this paper was the derivation of the It\^o formula for sticky diffusions. It is used in \cite{berry2024stationary} to prove a verification theorem in the context of an optimal control problem of sticky diffusions on metric graphs. Nevertheless, we believe that our results can be useful in various applications. One example can be found the recent study of Mean Field Game (MFG) problems on networks. We briefly recall that the theory of MFG was introduced independently in \cite{HMC2006} and \cite{LL2007} and aims at studying Nash equilibria of large symmetric dynamic games. General references on this theory are \cite{BFY2013,GS2014,CD2018,ACDPS2020}. Adaptation of the theory to metric graphs was first proposed in \cite{camilli} and further developed in \cite{CM2016,ADLT2019,ADLT2020,CM2024}. In all of these references agents are represented by nonsticky diffusions. In contrasts, the recent paper \cite{berry2024stationary} mentioned above deals with games were players present a sticky behavior at vertices. MFG of nonsticky diffusions were recently used in \cite{CFM2024} for a model of urban planning. Finally we also mention that another application of nonsticky diffusions to spatial economics can be found in \cite{AM2022}. We believe that considering nonsticky diffusions would be an interesting extension of these models.\\

The rest of the paper is structured as follows. Our main results are stated in \cref{section:statement} together with the necessary notations and definitions. \cref{section:preliminaries} contains a summary of the results from \cite{FreidlinWentzell,FreidlinSheu} that will be useful in our study of sticky diffusions on $\Gamma$ as well as some technical lemmas on these processes. \cref{section:main_results} contains the proofs of our main results and we derive the representation formulas in \cref{section:representation_formulas}. Finally \cref{section:BVP} contains a technical result about a system of ODEs which is used in \cref{section:preliminaries}.

\section{Notations and statement of main results}
\label{section:statement}

In order to state our main results, we now introduce some notations and definitions. Let $N$ be a positive natural number and let $E$ denote a family of copies of the positive half-line, defined as the disjoint union
\begin{align*}
{E}:= \bigsqcup_{i=1}^N [0,\infty).
\end{align*}
For notational simplicity, we represent each point in $E$ by a pair $(j,x)$, where $j \in \{1,\dots,N\}$ and $x$ represents the Euclidean distance from the origin. Following the approach in \cite{mugnolo2019actually}, we define an equivalence relation on $E$ as follows:
\begin{align*}
(j,x) \sim (i, y) \ \text{ if and only if } \ \begin{cases}
j=i \ \text{and } x=y\\
x=y=0, \ \text{for any } j,i.
\end{cases}
\end{align*}
Then, the star graph is the quotient space $\Gamma:= {E}/ \sim$. In this graph, there is a unique vertex $\sfv=(\cdot,0)$ that is common to all edges. We emphasize that with this configuration, we do not consider the angles between the edges, allowing it to maintain its abstract nature. Moreover, $\Gamma$ forms a metric space once provided with the following metric:
	\begin{align}
	\label{distance}
	d((j,x),(i,y)):=\begin{cases}
	\lvert x - y \rvert \quad & \textnormal{if } i = j, \\
	x + y \quad & \textnormal{if } i \neq j.
	\end{cases}
	\end{align}

	\begin{prop}
	 The metric space $(\Gamma,d)$ is a locally compact Polish space.
	\end{prop}
	\begin{proof}
	 We have to prove that $(\Gamma,d)$ is a locally compact complete and separable metric space. The fact that it is locally compact follows from the fact that $\RR_+$ is locally compact and separable. In order to prove the completeness of the metric, let $(\sfx_n)_{n \in \NN}$ be a Cauchy sequence in $\Gamma$. Since $\Gamma$ has finitely many edges, there exists a strictly increasing function $\varphi \colon \NN \to \NN$ and $i \in \{1, \dots, N \}$ such that $\sfx_{\varphi(n)} = (i, x_{\varphi(n)})$ for all $n \in \NN$. Since $(\sfx_n)_{n \in \NN}$ is a Cauchy sequence in $\Gamma$, it follows from the definition of the metric $d$ that $x_{\varphi(n)}$ is a Cauchy sequence in $\RR_+$. We then deduce from the completeness of $\RR_+$ that there exits $\bar x \in \RR_+$ such that $(i, x_{\varphi(n)}) \xrightarrow{n \to \infty} (i,\bar x) =: \bar \sfx$. Hence, $(x_n)_{n \in \NN}$ admits a convergent subsequence. It then follows from standard properties of Cauchy sequences that the full sequence must actually converge to $\bar \sfx$.
	\end{proof}

Since $(\Gamma,d)$ is a metric space, a function $f: \Gamma \to \mathbb{R}$ is considered continuous if it satisfies the standard definition of continuity for functions defined on metric spaces  and we denote by $C(\Gamma)$ the space of continuous functions on $\Gamma$ and $C_0(\Gamma)$ the space of continuous functions vanishing at infinity on each edge. We provide both $C(\Gamma)$ and $C_0(\Gamma)$ with the topology of uniform convergence on compact sets. We can represent a function $f \colon \Gamma \to \RR$ as
\begin{align*}
f=\bigoplus_{i=1}^N f_i \quad  \textnormal{satisfying } f_i(0) = f_j(0) \textnormal{ for all } i,j \in \{ 1, \dots, N \},
\end{align*}
where the notation $\bigoplus_{i=1}^N f_i$ stands for $(f_1, \dots, f_N)$ and each $f_i(\cdot) = f(i, \cdot): \RR_+ \to \mathbb{R}$ corresponds to the function restricted to the edge $i$. With this representation we have
\[
C(\Gamma)=\left\{f=\bigoplus_{i=1}^N f_i \in \bigoplus_{i=1}^N  C\left([0,\infty) \right) \text{ and } f_i(0)=f_j(0), \, \forall i, j \in \{1,\dots,N\}\right\},
\]
where $\bigoplus_{i=1}^N  C\left([0,\infty) \right)$ denotes the external direct sum of $N$ copies of $ C\left([0,\infty) \right)$.

Regarding differentiation, derivatives outside the vertex are viewed as, for $\mathsf{x}=(i,x) \in \Gamma \setminus \{\sfv\}$
\begin{align*}
f^\prime(\mathsf{x})= \frac{d}{dx} f_i (x).
\end{align*}
The same applies to higher order derivatives as well. We then define, for each $k \in \mathbb{N}$,
\[
C^k(\Gamma):=\left\{f=\bigoplus_{i = 1}^N f_i \in \bigoplus_{i=1}^N C^k\left([0,\infty) \right) :\,  f \in C(\Gamma) \right\} .
\]
Notice that, in the definition of $C^k(\Gamma)$, only the function is assumed to be continuous at the vertex and the derivatives may be discontinuous. We will also need to consider the following spaces :
\[
C^k_b(\Gamma) = \left \{ f \in C^{k}(\Gamma) :\, f^{(h)}_i \in C_b \left([0,\infty) \right) \, \textnormal{for all } h \in \{0, \dots, k\} \textnormal{ and } i \in \{1, \dots, N \} \right \}
\]
and, for every $k \in \NN$,
\[
PC^k(\Gamma) = \bigoplus_{i=1}^N C_b^k([0,+\infty)),
\]
with the convention that $PC^0(\Gamma) = PC(\Gamma)$. With an abuse of notations, we will say that a function $f \colon \Gamma \to \RR$ belongs to $PC^k(\Gamma)$ if each $f_i, \dots , f_i^{(k)}$ belongs to $C_b((0,+\infty))$ and can be continuously extended to $[0,+\infty)$.
Finally, for $k \geq 0$, we denote by $C_b^{1,k}(\RR_+ \times \Gamma)$ the space of function $f \colon \RR_+ \times \Gamma \ni (t,\sfx) \mapsto f(t,\sfx) \in \RR$ such that $f, \, \partial_t f \in C_b(\RR_+ \times \Gamma)$ and $\partial_x f_i, \dots,\partial_x^k f_i \in C_b(\RR_+ \times \RR_+)$ for each $i \in \{1, \dots, N \}$. A comprehensive treatment of functional spaces and equations on the graph is available in \cite{mugnolo-book}. \\

We now turn to defining diffusions on the graph. Let $L_i$, for $i=1,\dots, N$, be the linear operators defined on $PC^2(\Gamma)$ by
\begin{align}
\label{generator}
L_i f(x)= \frac{1}{2} \sigma_i^2 (x) f_i^{\prime \prime} (x) + b_i(x) f_i^\prime (x), \quad x>0,
\end{align}
where $\sigma, b \in PC(\Gamma)$ and we assume that
\begin{equation}\label{eq_diffusion_positive}
 \textnormal{there exists a constant $\sigma_0>0$ such that $\sigma(x)>\sigma_0$ for all $\sfx \in \Gamma$.}
\end{equation}
	Observe that, for every $f \in C_b^2(\Gamma)$ with $Lf \in C(\Gamma)$, i.e. such that
	\begin{equation}\label{eq:continuity_generator}
	\lim_{x \to 0} L_i f(x) = \lim_{y \to 0} L_j f(y) \quad \textnormal{for every } i,j \in \{ 1, \dots, N \},
	\end{equation}
	the operator defined in \eqref{generator} is an operator on the graph $\Gamma$. Indeed, given a point $\mathsf{x}=(i,x) \in \Gamma$, we alternatively write
	\begin{align}
	\label{generator2}
	L f(\mathsf{x})=\frac{1}{2} \sigma^2 (\sfx) f^{\prime \prime} (\sfx) + b(\sfx) f^\prime (\sfx), \quad \textnormal{for all }\mathsf{x} \in \Gamma,
	\end{align}
	where \eqref{eq:continuity_generator} ensures that the expression \eqref{generator2} makes sense for $\mathsf{x} = \sfv$.
In order to obtain a process on the graph, we need a transition condition on the vertex, which, in our case, takes the form
	\begin{align}
	\label{kirchoff}
	\eta Lf(\sfv) = \sum_{i=1}^{N} \rho_i f^\prime_i(0),
	\end{align}
where $\eta \geq 0$,  $\rho_i > 0$ for all $i \in \{1 ,\dots, N \}$ and $\sum_{i=1}^{N} \rho_i = 1$. In the end, the generator of a diffusion process on $\Gamma$ is given by $(L,D(L))$ where $L$ is defined by \eqref{generator2} and
\begin{equation}\label{eq:domain}
D(L) = \left \{ C^2(\Gamma) \cap C_0(\Gamma) :\, Lf \in C_0(\Gamma) \, \textnormal{and } \eta Lf(\sfv) =\sum_{i=1}^{N} \rho_i f^\prime_i(0)  \right \},
\end{equation}
The process will be called a \emph{nonsticky diffusion} if $\eta =0$ and a \emph{sticky diffusion} otherwise. The constant $\eta$ will be called the \emph{stickiness parameter} of the process. \\

We are now able to state our main results.

\begin{tm*}
	Let $X(t) = (i(t), x(t))$ be a sticky diffusion on a star graph $\Gamma$ with stickiness parameter $\eta > 0$. Then,
	\begin{enumerate}[label={\rm(\roman*)}]
	 \item \label{item:characterization} there exists a nonsticky diffusion $Y$ on $\Gamma$ and a strictly increasing continuous function $V^{-1}$ such that $X$ and $Y(V^{-1}(\cdot))$ are indistinguishable. Moreover, there exists an increasing process $\ell_X$, adapted to the completion of the natural filtration of $X$ with respect to the collection of probability measures $\left ( \mathbf P_\sfx \right)_{\sfx \in \Gamma}$, such that $\ell_X$ almost surely increases only when $X = \sfv$ and $V^{-1}(t) = t - \eta \ell_X(t)$;

	 \item \label{item:SDE} up to an extension of the filtered probability space, there exists a one-dimensional Brownian motion $W$ such that, almost surely, we have
	 \begin{equation} \label{sde:sticky_intro}
		\begin{cases}
		dx(t) = \sigma_{i(t)}(x(t)) \mathbf{1}_{\{x(t) \ne 0\}}\, d W(t) + b_{i(t)}(x(t)) \mathbf{1}_{\{x(t) \ne 0\}}\, dt + d\ell_X(t), \\
		\eta \ell_X(t) = \int_0^t \mathbf 1_{\{X(s) = \sfv \}} ds;
		\end{cases}
	\end{equation}

	\item \label{item:Ito} for every $f \in C_b^{1,2}(\RR_+ \times \Gamma)$, almost surely, we have the It\^o formula
	\begin{equation}\label{eq:Ito_1_intro}
	\begin{split}
	& f(t,X(t)) = f(0,X(0)) + \int_0^t \left ( \partial_s f(s,X(s)) +  Lf(s,X(s)) \right ) \mathbf 1_{\{ X(s) \neq \sfv \}}  ds \\
	& \quad + \int_0^t \sigma(X(s)) \partial_x f(s,X(s)) \mathbf{1}_{\{X(s) \neq \sfv \}} d  W(s)
	 + \int_0^t \left ( \eta\, \partial_s f(s,\sfv) +  \sum_{i=1}^N \rho_i \partial_x f_i(s,0) \right ) d \ell_X(s).
	\end{split}
	\end{equation}
	\end{enumerate}
\end{tm*}

\begin{proof}
 The different statements are proved in \cref{section:main_results}. More precisely \ref{item:characterization} is established in \cref{tm:characterization}, \ref{item:SDE} follows from \cref{tm:SDE,prop:occupation} and \ref{item:Ito} is proved in \cref{tm:Ito}.
\end{proof}

Notice that one may recover results about sticky and skew sticky diffusions discussed above by considering the spacial cases $N=1$ and $N=2$ respectively. \\

\section{Diffusion processes on $\Gamma$}
\label{section:preliminaries}

We start by summarizing the conclusions from \cite[Theorem 3.1]{FreidlinWentzell} and \cite[Lemma 2.2, Remark 2.5]{FreidlinSheu}.

\begin{tm}\label{tm:FW_FS}
 The unbounded linear operator $L$, defined by \eqref{generator2} with domain given by \eqref{eq:domain}, is the generator of a strongly continuous semigroup on $C_0(\Gamma)$, associated to a conservative Markov process $X$ on $\Gamma$ having continuous paths. Moreover, writing $X(t)= (i(t), x(t))$, where $x(t)$ is the diffusive part on $[0,\infty)$ generated by $L_i$ and $i(t)$ selects the edges and denoting by $(\cF_t)_{t \geq 0}$ the natural filtration associated to $X$, we have the following:
\begin{enumerate}[label={\textnormal{(\roman*)}}]
 \item \label{item:FW} Let $t \geq 0$, and $\tau := \inf \{ s > t : x(s) = 0 \}$. Then $x$ coincides with a one dimensional diffusion on $[0,\infty)$, with generator $L_{i(t)}$, on $[t,\tau)$.
 
 \item \label{item:FS} If $\eta = 0$, then we have
 \[
  \int_0^t \mathbf{1}_{\{x(s) = 0 \}}\, ds = 0 \quad \textnormal{for all } t \geq 0,
 \]
 and there exists a one-dimensional Brownian motion $W$ and a continuous increasing process $\ell_X$, both adapted to $(\cF_t)_{t \geq 0}$, such that 
 \begin{equation}\label{eq:freidlin-sheu-sde}
  dx(t) = \sigma_{i(t)}(x(t)) dW(t) + b_{i(t)}(x(t)) dt + d\ell_X(t),
 \end{equation}
 and the process $\ell_X$ increases only when $x(t) = 0$. Furthermore, we have
 \[
   \left ( \sum_{i=1}^N \frac{2 \rho_i}{\sigma_i^2(0)} \right) \ell_X(t) = \lim_{\delta \to 0} \frac{1}{\delta} \int_0^t \mathbf 1_{\{x(s) \leq \delta \}} \, ds
 \]
 in expectation\footnote{\textit{i.e.}, in $L^1$.}.
\end{enumerate}

\end{tm}

To complete the notation, we denote by $\mathbf{E}_\mathsf{x}$ the expected value with respect to $\mathbf{P}_\mathsf{x}$, where $\mathsf{x}$ is the starting point. We continue with a simple observation. First, we recall, \cite[Proposition 2.19, Chapter III]{revuz-yor}, that \(\mathsf{x} \in \Gamma\) is a trap point for the process \(X\) if
\[
\mathbf{P}_\mathsf{x}(\xi_\mathsf{x}>t)=1,
\]
for each \(t>0\) and \(\xi_\mathsf{x}:=\inf\{s >0 : X(s) \ne \mathsf{x}\}.\)
\begin{prop}\label{prop:trap}
 The vertex $\sfv$ is not a trap for the Makov process generated by \eqref{generator2} and \eqref{eq:domain}, for a finite $\eta \geq 0$.
\end{prop}
\begin{proof}
 Arguing by contradiction, assume that $\sfv$ is a trap for the Markov process $X$ generated by \eqref{generator2}-\eqref{eq:domain}. Denote by $(P_t)_{t \geq 0}$ the strongly continuous semigroup of bounded linear operators on $C_0(\Gamma)$ associated to $X$. Then $P_t g(\sfv) = \mathbf E_\sfv[g(X(t))] = g(\sfv)$ for all $g \in C_0(\Gamma)$ and $t \geq 0$. Let $u \in D(L)$ be such that $Lu(\sfv) = 1$, such a function may be obtained by extending the one obtain in \cref{prop:bvp} with $f = 1$. Then, using Kolmogorov's forward equation, we have
 \[
  0 = \frac{d}{dt} P_t u(\sfv) = P_t (L u)(\sfv) = 1,
 \]
 a contradiction.
\end{proof}

The following lemma provides an asymptotic estimate on the mean first passage time of the process started at $\sfv$ at the boundary of the ball $B(\sfv,\delta)$ as $\delta$ tends to zero.
\begin{lemma}\label{lem:MFPT}
 Let $X(t) = (i(t),x(t))$ be a Markov process with generator given by \eqref{generator2} and \eqref{eq:domain} for $\eta \geq 0$ and, for every $\delta > 0$, set
 \begin{equation}\label{eq:MFPT}
  T_X^\delta:=\inf\{t \geq 0 : x(t) \geq \delta\}.
 \end{equation}
 Then $\left \lvert \mathbf E_{\mathsf v}[T_X^\delta] -\eta \delta \right \rvert = O(\delta^2)$ when $\delta \downarrow 0$.

\end{lemma}
\begin{proof}
	For each $\delta > 0$ and $1 \leq i \leq N$, consider $u^\delta \in D(L)$ such that the restriction of $u^\delta$ to $B(\mathsf v, \delta)$ satisfies the system of differential equations
	\begin{equation}
		\begin{cases}
		 \frac{1}{2} \sigma^2_i(x) {u^\delta_i}''(x) + b_i(x){u^\delta_i}'(x) = -1 \quad & \textnormal{for all } x \in (0, \delta), \, 1 \leq i \leq N, \\
		 \eta Lu(\sfv) - \sum_{i=1}^N \rho_i {u^\delta_i}'(0) = 0, \\
		  u^\delta_i(\delta) = 0 \quad & \textnormal{for all } i \in \{1,\dots, N \},
		\end{cases}
	\end{equation}
	and 
	\begin{equation}\label{eq:mfpt_estimate}
	 \sup_{0 < d(\sfx, \sfv) < \delta} \left \lvert {u^\delta}''(\sfx) \right \rvert = O(1) \quad \textnormal{as } \delta \downarrow 0.
	\end{equation}
	Such a function can be obtained by extending the function obtained in the postponed \cref{prop:bvp}, with $f_i(x) = -1$ for all $0 \leq x < \delta$.
	Then, using Dynkin's formula, we have
	\begin{align*}
	u_i^\delta(0) & = \mathbf E_{\sfv} \left [u^\delta(0) - u^\delta(X(T_X^\delta)) \right ] = \mathbf E_{\sfv} \left[ - \int_0^{T_X^\delta} Lu^\delta(X(s)) \, ds \right ] = \mathbf E_{\sfv} \left[ T_X^\delta \right].
	\end{align*}
	On the other hand, using Taylor's formula, there exists $(\theta_1, \dots, \theta_N) \in (0,\delta)^N$ such that
	\begin{align*}
	 u_i^\delta(0) = \sum_{i=1}^N \rho_i \left (u^\delta_i(0) - u^\delta_i(\delta) \right) & = - \sum_{i=1}^N \rho_i \left ({u^\delta_i}'(0) \delta + {u^\delta_i}''(\theta_i) \frac{\delta^2}{2} \right) = \eta \delta - \sum_{i=1}^N \rho_i {u^\delta_i}''(\theta_i) \frac{\delta^2}{2}.
	\end{align*}
	We conclude that
	\[
	 \left \lvert \mathbf E_{\mathsf v}[T_X^\delta] -\eta \delta \right \rvert \leq \sup_{0 < d(\sfx, \sfv) < \delta} \left \lvert {u^\delta}''(\sfx) \right \rvert \delta^2.
	\]
	The conclusion then follows from \eqref{eq:mfpt_estimate}.
\end{proof}

The following result was established in \cite[Corollary 2.4]{FreidlinSheu} in the case $\eta = 0$. We provide here a generalization of this result the case to $\eta \geq 0$ with a slightly different proof.

\begin{lemma}\label{lem:kirchhoff}
 Let $X(t) = (i(t),x(t))$ be a Markov process with generator given by \eqref{generator2} and \eqref{eq:domain} for $\eta \geq 0$ and, for all $\delta > 0$, define $T_X^\delta$ according to \eqref{eq:MFPT}.
 Then 
 \[
  \rho_j = \lim_{\delta \downarrow 0} \mathbf{P}_\mathsf{v} \left ( i(T_X^\delta) = j \right) \quad \textnormal{for all } 1 \leq j \leq N,
 \]
where $(\rho_j)_{1 \leq j \leq N}$ are the constants appearing in \eqref{eq:domain}.
\end{lemma}

\begin{proof}
 Let $f \in D(L)$ and set $\tilde \rho^\delta_j := \mathbf{P}_\mathsf{v} \left ( i(T_X^\delta) = j \right)$ for all $1 \leq j \leq N$ and $\delta > 0$. 
 Using Taylor's formula, we have
	\begin{equation}\label{eq:kirchhoff_1}
	 \mathbf{E}_\mathsf{v} \left [f(X(T_X^\delta)) - f(\mathsf{v}) \right ] = \sum_{j=1}^N \tilde \rho^\delta_j \left (f_j(\delta) - f_j(0) \right) = \sum_{j=1}^N \tilde \rho_j^\delta \left (f_j'(0) \delta + o(\delta) \right).
	\end{equation}
 On the other hand, from Dynkin's formula, we also have 
 \begin{align*}
 &\mathbf{E}_\mathsf{v} \left [\delta^{-1} \left (f(X(T_X^\delta)) - f(\mathsf{v}) \right ) \right ]  = \mathbf{E}_\mathsf{v} \left [ \frac{1}{\delta} \int_0^{T_X^\delta} L f(X(s)) \, ds \right ] \\
 & \qquad = \frac{\mathbf E_{\sfv} \left [T_X^\delta \right ]}{\delta} \left ( Lf(\sfv) + \mathbf{E}_\mathsf{v} \left [ \frac{1}{\mathbf E_{\sfv} \left [T_X^\delta \right ]}\int_{0}^{T_X^\delta} \left( L f(X(s)) - Lf(\sfv) \right) \, ds \right ] \right ),
\end{align*}
and notice that 
\begin{equation}\label{eq:vanishing_integal}
 \left \lvert \mathbf{E}_\mathsf{v} \left [ \frac{1}{\mathbf E_{\sfv} \left [T_X^\delta \right ]}\int_{0}^{T_X^\delta} \left( L f(X(s)) - Lf(\sfv) \right) \, ds \right ] \right \rvert \leq \sup_{\sfx , \mathsf y \in B(\sfv, \delta)} \left \lvert Lf(\sfx) - Lf(\mathsf y) \right \rvert \xrightarrow{\delta \to 0} 0.
\end{equation}
Let $(\delta_n)_{n \in \NN}$ be any sequence in $(0,\infty)$ converging to $0$ and such that $\tilde \rho_j^{\delta_n}$ converges to some $\tilde \rho_j$ for each $j$. Then, passing to the limit $n \to \infty$ and using the fact that $\left \lvert \mathbf E_{\mathsf v}[T_X^\delta] -\eta \delta \right \rvert = o(\delta)$ (see \cref{lem:MFPT}), we obtain that 
\[
 \eta L f(\sfv) = \sum_{j=1}^N \tilde \rho_j f_j'(0).
\]
Since the latter holds for every $f \in D(L)$, we conclude that $\tilde \rho_j = \rho_j$.
\end{proof}

\begin{lemma}\label{lem:occupation}
 Let $X(t) = (i(t),x(t))$ be a Markov process with generator given by \eqref{generator2} and \eqref{eq:domain} for $\eta \geq 0$ and, for every $\delta > 0$, define the stopping time $T_X^\delta$ by \eqref{eq:MFPT}.
 Then
 \[
  \left \lvert \mathbf E_{\sfv} \left[ \int_0^{T_X^\delta} \mathbf 1_{\{X(s) = \sfv \}} \, ds \right ] -\eta \delta \right \rvert = O(\delta^2) \quad \textnormal{when } \delta \downarrow 0.
 \]
\end{lemma}

\begin{proof}
	For each $\delta > 0$, consider a sequence $f^n \in \bigoplus_{i=1}^N C([0,\delta])$ satisfying
	\[
	0 \leq f_i^n \leq 1  \quad \textnormal{and} \quad f_i^n(x) = \begin{cases}
	          1 \quad & \textnormal{for } 0 \leq x \leq \frac{\delta}{3n}, \\
	          0 \quad & \textnormal{for } \frac{2\delta}{3n} \leq x \leq \delta,
	         \end{cases}
	\]
	for all $1 \leq i \leq N$, and, as in the proof of \cref{lem:MFPT}, a sequence of functions $u^n \in D(L)$ such that the restriction of $u^n$ to $B(\mathsf v, \delta)$ satisfies
	\begin{equation}
		\begin{cases}
		 \frac{1}{2} \sigma^2_i(x) {u^n_i}''(x) + b_i(x){u^n_i}'(x) = - f_i^n(x) \quad & \textnormal{for all } x \in (0, \delta), \, 1 \leq i \leq N, \\
		 \eta Lu^n(\sfv) - \sum_{i=1}^N \rho_i {u^n_i}'(0) = 0, \\
		  u^n_i(\delta) = 0 \quad & \textnormal{for all } i \in \{1,\dots, N \}.
		\end{cases}
	\end{equation}
	From Dynkin's formula, we have
	\begin{align*}
	u^n(\sfv) & = \mathbf E_{\sfv} \left [u^n(\sfv) - u^n(X(T_X^\delta)) \right ] = \mathbf E_{\sfv} \left[ - \int_0^{T_X^\delta} Lu^n(X(s)) \, ds \right ] \\
	& = \mathbf E_{\sfv} \left[ \int_0^{T_X^\delta} f^n(X(s)) \, ds \right ]
	\end{align*}
	On the other hand, using Taylor's formula, there exists $(\theta_1^n, \dots, \theta_N^n) \in (0,\delta)^N$ such that
	\begin{align*}
	 u^n(\sfv) &= \sum_{i=1}^N \rho_i \left (u^n_i(0) - u^n_i(\delta) \right) = - \sum_{i=1}^N \rho_i \left ({u^n_i}'(0) \delta + {u^n_i}''(\theta_i^n) \frac{\delta^2}{2} \right) \\
	 & = - \eta L u^n(\sfv) \delta - \sum_{i=1}^N \rho_i {u^n_i}''(\theta_i^n) \frac{\delta^2}{2} \\
	 & = \eta \delta - \sum_{i=1}^N \rho_i {u^n_i}''(\theta_i^n) \frac{\delta^2}{2}
	\end{align*}
	We conclude that
	\begin{equation}\label{eq:occupation_1}
	 \left \lvert \mathbf E_{\sfv} \left[ \int_0^{T_X^\delta} f^n(X(s)) \, ds \right ] -\eta \delta \right \rvert \leq \sup_{0 < d(\sfx, \sfv) < \delta} \left \lvert {u^n}''(\sfx) \right \rvert \frac{\delta^2}{2}.
	\end{equation}
	Finally, noticing that $f^n$ converges point-wise to $\mathbf 1_{\{\sfv\}}$ and $n \to \infty$ and using \cref{prop:bvp} to bound the right-hand side of \eqref{eq:occupation_1}, we deduce from Lebesgue's dominated convergence theorem that, for every $\delta_0 > 0$, there is a constant $C_0 > 0$ such that
	\[
	 \left \lvert \mathbf E_{\sfv} \left[ \int_0^{T_X^\delta} \mathbf 1_{\{X(s) = \sfv \}} \, ds \right ] -\eta \delta \right \rvert \leq C_0 \delta^2.
	\]
	for all $0 < \delta \leq \delta_0$. This concludes the proof.
\end{proof}

\section{Main results}
\label{section:main_results}

\subsection{Construction of trajectories of sticky diffusions}

As mentioned in the introduction, sticky conditions are characterized by a slowdown which, in probabilistic terms, is given by a time-change of the process involving the local time. Our first result concerns the extension of this idea to metric graphs for a general diffusion.

Let $Y(t)=(j(t), y(t))$ be the nonsticky process generated by \eqref{generator2} and \eqref{eq:domain} for $\eta = 0$, and denote by $(\cG_t)_{t \geq 0}$ the natural filtration associated to $Y$.
As for the one dimensional case, we define the time-changed process
\begin{align}
\label{timechanged}
X(t)= Y (V^{-1}(t)),
\end{align}
where $V^{-1}$ is the right inverse of $V$, defined by
\begin{equation}\label{eq:time_change}
V^{-1}(t)= \inf \{s \geq 0\,:\, V(s) >t \}, \quad t>0
\end{equation}
and
\begin{align}
\label{time-change}
V(t)= t + \eta \ell_Y(t).
\end{align}
Since $V$ is continuous, we know from \cite[Proposition 4.6 p.43]{revuz-yor} that each $V^{-1}(t)$ is a stopping time for the natural filtration of $\ell_Y$. In particular it is a stopping time for $\cG_t$ because $\ell_Y$ is adapted to $\cG_t$.  Moreover, since $V$ is also strictly increasing, it is known that the process $V^{-1}$ is also strictly increasing and continuous and satisfies $V^{-1}(V(t)) = t$ almost surely. Consequently, it is a time change in the sense of \cite[Definition 1.2 p.180]{revuz-yor}. In particular the process $X$ is adapted to the time-changed filtration
\[
 \cF_t := \cG_{V^{-1}(t)} := \left \{ A \in \cG_\infty : A \cap \{ V^{-1}(t) \leq s \} \in \cG_s \quad \textnormal{for all } s \geq 0 \right \}.
\]
It is known that $X$ is a Markov process for the filtration $\cF_t$ (see \cite[Theorem 2.18 p.417]{revuz-yor} for this result in the case of a Brownian motion and \cite[Theorem 10.10]{D1965} for the general case). We will prove that this process is a version of the sticky diffusion on $\Gamma$, with stickiness parameter $\eta$. This is achieved by a characterization of the generator of the process $X$.

\begin{prop}\label{tm:time-change}
	Let $Y(t) = (j(t),y(t))$  be the nonsticky process generated by \eqref{generator2} and \eqref{eq:domain} for $\eta = 0$ and define the time-changed $X$ according to \eqref{timechanged}. Then, the semigroup induced by $X$ on $C_0(\Gamma)$ is generated by \eqref{generator2} and \eqref{eq:domain}, where $\eta$ is the same constant appearing in \eqref{time-change}. In particular $X$ is a version of the Markov process given by \eqref{generator2} and \eqref{eq:domain}. Moreover, $V^{-1}(t) = t - \eta \ell_X(t)$, where $\ell_X(t) := \ell_Y(V^{-1}(t))$ and $\ell_Y$ is the process appearing in \cref{tm:FW_FS}-\ref{item:FS}.
\end{prop}
\begin{proof}

	 Let $f \in D(L^\eta) \cap C^2_b(\Gamma)$, with $\eta > 0$, where $L^\eta$ denotes the generator of the semigroup induced by $X$ on $C_0(\Gamma)$ and $D(L^\eta)$ is its domain. One may check that $L^\eta f(\sfx) = Lf(\sfx)$ for every $\sfx \neq \sfv$ (see \cref{tm:characterization} below for a similar argument). We therefore turn to the characterization of the junction condition.
	 
	 Fix $\delta>0$, we recall the definition of the stopping time
		\begin{align*}
		T_X^\delta:=\inf\{t \geq 0 : x(t) \geq \delta\}.
		\end{align*} 
	  From Dynkin's formula we know
		\begin{equation}\label{dynkin}
		 \mathbf{E}_{\sfv}[f(X_{T_X^\delta} ) ] - f(\sfv) = \mathbf{E}_\sfv \left [\int_0^{T_X^\delta} L f(X(s)) \, ds \right ].
		\end{equation}
		Setting $\rho_j^\delta := \mathbf{P}_\mathsf{v} \left ( i(T_X^\delta) = j \right)$ for each $1 \leq j \leq N$, we have
		\begin{align}
		\label{dynkin:numerator}
		\mathbf{E}_{\sfv}[f(X_{T_X^\delta} ) ] - f(\sfv) = \sum_{k=1}^{n} \rho^\delta_k f_k(\delta)-f(\sfv)= \sum_{k=1}^{n} \rho^\delta_k (f_k(\delta) - f_k(0)).
		\end{align}
		So that 
		\begin{equation}\label{eq:dynkin_2}
		 \mathbf{E}_\sfv \left [\int_0^{T_X^\delta} L f(X(s)) \, ds \right ] = \sum_{k=1}^{n} \rho^\delta_k (f_k(\delta) - f_k(0)).
		\end{equation}
		We introduce the analogous hitting time for the nonsticky diffusion $Y$
		\begin{align*}
		T_Y^\delta:=\inf\{t \geq 0 : y(t) \geq \delta\}.
		\end{align*}
		Our claim is that $T_X^\delta=V(T_Y^\delta)$, where $V$ is the process defined in \eqref{time-change}. Since $V^{-1}$ is its inverse and it is a continuous process, trivially we have $V^{-1} (V(T_Y^\delta))= T_Y^\delta$ and $y(V^{-1} (V(T_Y^\delta)))=\delta$. Recalling that $X(t)=Y(V^{-1} (t))$, we obtain the inequality $T_X^\delta \leq V(T_Y^\delta)$, which we do not know a priori to be an equality as it might not be the infimum. If we assume that the last inequality is strict, then by composing with $V^{-1}$ we have
		\begin{align*}
		V^{-1} (T_X^\delta) < V^{-1}(V(T_Y^\delta))=T_Y^\delta,
		\end{align*}
		since $V^{-1}$ is strictly increasing. But $y(V^{-1} (T_X^\delta))= x(T_X^\delta)=\delta$, in contradiction with the definition of hitting time $T_Y^\delta$. We conclude that
		\begin{align*}
		T_X^\delta = V(T_Y^\delta)= T_Y^\delta + \eta \ell_Y(T_Y^\delta).
		\end{align*}
		Taking expectations, we get
		\begin{align*}
		\mathbf{E}_{\sfv}[T_X^\delta]=\mathbf{E}_{\sfv}[T_Y^\delta]+\eta \mathbf{E}_{\sfv}[\ell_Y(T^\delta_Y)].
		\end{align*}
		From \cref{lem:MFPT}, we have $\mathbf{E}_{\sfv}[T_Y^\delta]=O(\delta^2)$ and, using also \eqref{eq:freidlin-sheu-sde} and the fact that $b$ is bounded, we have
		\[
		 \mathbf{E}_{\sfv}[\ell_Y(T^\delta_Y)]=\mathbf{E}_{\sfv}[y(T^\delta_Y)] + \mathbf{E}_{\sfv} \left [ \int_0^{T_{Y}^\delta} b_{j(s)}(y(s)) ds \right ] = \delta + O(\delta^2) \quad \textnormal{as } \delta \downarrow 0.
		\]
		Moreover, from \eqref{eq:dynkin_2}, we deduce
		\begin{equation*}
		 \frac{\mathbf E_\sfv \left [ T_X^\delta \right ]}{\delta} \left ( Lf(\sfv) + \mathbf{E}_\sfv \left [\frac{1}{\mathbf E_\sfv \left [ T_X^\delta \right ]} \int_0^{T_X^\delta} \left( L f(X(s)) - Lf(\sfv) \right) \, ds \right ] \right ) = \sum_{k=1}^{n} \rho^\delta_k \frac{f_k(\delta) - f_k(0)}{\delta}.
		\end{equation*}
		Using \eqref{eq:vanishing_integal} together with \cref{lem:MFPT,lem:kirchhoff}, we conclude by passing to the limit $\delta \downarrow 0$ that
		
			\begin{align*}
			\eta L f(\sfv)= \sum_{k=1}^{n} \rho_k f^\prime_k (0).
			\end{align*}
		This proves that $D(L^\eta) \cap C^2_b(\Gamma) \subset D(L)$. It then follows that the semigroups induced by $(L,D(L))$ and $(L^\eta,D(L^\eta))$ coincide.

		Finally, we define
		\begin{equation}\label{eq:time_change_local_time}
		\ell_X(t) = \ell_Y(V^{-1}(t)) \quad \textnormal{a.s. for all } t \geq 0.
		\end{equation}
		Notice that
		\[
		t = V(V^{-1}(t)) = V^{-1}(t) + \eta \ell_Y(V^{-1}(t)) = V^{-1}(t) + \eta \ell_X (t),
		\]
		so that
		\begin{equation}\label{eq:time_change_representation}
		V^{-1}(t) = t - \eta \ell_X(t).
		\end{equation}
	\end{proof}

	From Theorem \ref{tm:time-change} we are able, similarly to \cite[Theorem 3.1]{salins}, to characterize the occupation time of the sticky diffusion process at $\sfv$.

	\begin{prop}\label{prop:occupation}
		For every bounded and measurable function $f \colon \RR_+ \to \RR$ the time-changed process $X$ satisfies
		\begin{equation}\label{eq:integral_local_time}
		\int_0^t f(s) \mathbf{1}_{\{X(s) = \sfv\}}\, ds = \eta \int_0^t f(s) d\ell_X(s) \quad \textnormal{almost surely for every } t > 0,
		\end{equation}
		where the integral in the right-hand side must be understood in the Lebesgue-Stieltjes sense. In particular, the occupation time of $X$ at $\sfv$ is given by
		\begin{align}
		\label{simple-case}
		\int_0^t \mathbf{1}_{\{X(s) = \sfv\}}\, ds = \eta \ell_X(t) \quad \textnormal{almost surely for every } t > 0.
		\end{align}

	\end{prop}
	\begin{proof}
	 We start by proving \eqref{simple-case}. From \cref{tm:FW_FS}-(ii), we know that, since the diffusion $Y$ has $\eta=0$ in \eqref{eq:domain},
			\begin{align*}
			\int_0^t \mathbf{1}_{\{Y(s) = \sfv \}} ds =0 \quad \textnormal{almost surely for all } t > 0.
			\end{align*}
			In particular,
			\begin{align*}
			\int_0^{V^{-1}(t)} \mathbf{1}_{\{Y(s) = \sfv \}} ds =0 \quad \textnormal{almost surely for all } t > 0,
			\end{align*}
			and since $V^{-1}(t)$ is an increasing function (in particular it has finite variation) and using formula \eqref{eq:time_change_representation}, we rewrite the integral as the Lebesgue-Stieltjes integral
			\begin{align}
			\label{difference-integral-local-time}
			0&=  \int_0^{V^{-1}(t)} \mathbf{1}_{\{Y(s) = \sfv \}} ds = \int_0^{t} \mathbf{1}_{\{X(s) = \sfv \}} dV^{-1}(s) = \int_0^{t} \mathbf{1}_{\{X(s) = \sfv \}} d(s - \eta \ell_X(s)) \notag\\
			&=\int_0^{t} \mathbf{1}_{\{X(s) = \sfv \}} ds - \eta \int_0^{t} \mathbf{1}_{\{X(s) = \sfv \}} d \ell_X(s)\\
			&=\int_0^{t} \mathbf{1}_{\{X(s) = \sfv \}} ds - \eta \ell_X(t) \notag,
			\end{align}
			where, in the last step, we used the fact that the random measure $d \ell_X$ is almost surely supported on the set $\{t \geq 0 : X(t) = \sfv\}$. We now turn to the proof of \eqref{eq:integral_local_time}. Notice first that the conclusion is a direct consequence of \eqref{simple-case} if $f$ is a simple function. The general case then follows from the dominated convergence theorem, after approximation of $f$ by simple functions.
	\end{proof}

	\begin{remark}
		The diffusion $X$ we are considering is a Markov process, as guaranteed by \cite[Theorem 3.1]{FreidlinWentzell}. However, if we want to introduce memory to the process and make it non-Markovian, we can use recent developments provided by the Fractional Boundary Value Problem, as described in \cite[Section 4]{BonDov} and \cite{d2024fractional}. A complete treatment, on star graphs, is also available in \cite[Section 4]{boncoldovpag2024}.\\
		Let us denote by $H$ a subordinator characterized by
		\begin{align*}
		\mathbf{E}_0[e^{-\lambda H(t)}]= e^{-t \Phi(\lambda)}, \quad \lambda >0,
		\end{align*}
		where $\Phi$ is a Bernstein function
		\begin{align*}
		\Phi(\lambda)= \int_0^\infty (1-e^{-\lambda z}) \Pi^\Phi(dz), \quad \lambda>0,
		\end{align*}
		and  $\Pi^\Phi$ is a L\'evy measure on $(0,\infty)$ such that $\int_0^\infty (1 \wedge z) \Pi^\Phi(dz) < \infty$. For us,  $\overline{\Pi}^\Phi(z)=\Pi^\Phi(z,\infty)$ is the tail of the L\'evy measure.  We also introduce the process
		\begin{align*}
		V_H(t)=t+ H(\eta \ell_Y(t))
		\end{align*}
		and $V_H^{-1}$ is its right inverse. The most significant difference will be this dependence on $H$ in the time change.

		Let $M>0$, $w\geq 0$ and $\mathcal{M}_w$ be the set of (piecewise) continuous function on $[0, \infty)$ of exponential order $w$ such that $\vert \varphi(t)\vert \leq M e^{wt}$. Let $\varphi \in \mathcal{M}_w \cap C[0,\infty)$ with $\varphi^\prime \in \mathcal{M}_w$. Then, we define the Caputo-D\v{z}rba\v{s}jan-type derivative as
		\begin{align}
		\label{Caputo}
		\mathfrak{D}_{t}^\Phi \varphi(t):=\int_0^t \varphi^\prime (s) \overline{\Pi}^\Phi(t-s) ds
		\end{align}
		which is a convolution type operator. We refer to the operator in \eqref{Caputo} as a Caputo-D\v{z}rba\v{s}jan-type derivative, as it generalizes the classical Caputo-D\v{z}rba\v{s}jan derivative by replacing the specific tail of the L\'evy measure associated with $\Phi(\lambda) = \lambda^\alpha$, $\alpha \in (0,1)$, with a more general one. It was introduced independently by Caputo \cite{Caputo1969} and by D\v{z}rba\v{s}jan \cite{Dzhrbashyan1966}, hence the name.
		Then, by proceeding as in \cite[Theorem 10]{BonDov}, we have that $Y(V_H^{-1}(t))$ is related to the non-local dynamic condition
		\begin{align*}
		\eta \mathfrak{D}_{t}^\Phi u(t,\sfv) = \sum_{i=1}^n \rho_i \frac{d}{dx} u_i(t,x)\vert_{x=0}.
		\end{align*}
		In this case, the dynamics is no longer Markovian, since in the time change appears the inverse of the subordinator $H$, which is non-Markov. Then, it spends a non-exponential time in the vertex. Just as the process
		$X$ has its roots in the skew sticky Brownian motion, the process $Y(V_H^{-1})$ originates from the non-local skew sticky Brownian motion, described in \cite{colantoni2023non}.
	\end{remark}

	\subsection{Characterization of sticky diffusions}
	We now prove that any sticky process has a modification that is a time-changed nonsticky one.

	\begin{tm}\label{tm:characterization}
	 Let $X(t) = (i(t), x(t))$ be the process generated by \eqref{generator2} and \eqref{eq:domain} with $\eta > 0$. Then there exists a nonsticky diffusion $\tilde Y$ on $\Gamma$ and a strictly increasing continuous function $V^{-1}$ such that $X$ and $\tilde Y(V^{-1}(\cdot))$ are indistinguishable. Moreover, there exists an increasing process $\ell_X$, adapted to the completion of the natural filtration of $X$ with respect to the collection of probability measures $\left ( \mathbf P_\sfx \right)_{\sfx \in \Gamma}$, such that $\ell_X$ almost surely increases only when $X = \sfv$, that $\eta \ell_X(t) = \int_0^t \mathbf{1}_{\{X(s) = \sfv\}} \, ds$ and $\ell_X(t) = \ell_{\tilde Y}(V^{-1}(t))$ almost surely, where $\ell_{\tilde Y}$ is the process appearing in \cref{tm:FW_FS}-\ref{item:FS}.
	\end{tm}

	\begin{proof}
	The argument is reminiscent of part 4 of the proof of \cite[Theorem 5]{EP2014}.
	 Set
	 \[
	  R(t) := \int_0^t \mathbf 1_{\{X(s) \neq \sfv \}} \, ds = t - \int_0^t \mathbf 1_{\{X(s) = \sfv \}} \, ds,
	 \]
	 for $t \geq 0$ and notice that $R$ is a.s. continuous and nondecreasing. Furthermore, since the vertex $\sfv$ is not a trap when $\eta$ is finite according to \cref{prop:trap}, we deduce that $\lim_{t \to \infty} R(t) = \infty$ almost surely. We may therefore consider the right inverse
	 \[
	  R^{-1}(t) : = \inf \left \{ s > 0 : R(s) > t \right \} \quad \textnormal{for } t \geq 0,
	 \]
	 which is strictly increasing and right-continuous.
	 We claim that $R$ cannot be constant on some interval $[a,b] \subset \RR_+$ with $b > a$. Indeed, if it is the case, we have $\int_a^b \mathbf{1}_{\{X(s) \neq \sfv\}}\, ds = 0$ and therefore the set $\{t \in [a,b] : X(t) \neq \sfv \}$ is negligible. Since $X$ almost surely has continuous paths, this implies that $X(t) = \sfv$ a.s. on $[a,b]$, a contradiction. This implies that $R^{-1}$ is in fact continuous and that $R^{-1}(R(t)) = t$ almost surely.
	 Moreover, since $R$ is $\mathcal{F}_t$-adapted, where $\mathcal F_t$ is the natural filtration of $X$, we have that each $R^{-1}(t)$ is a $\mathcal F_t$-stopping time. It follows that the process $R^{-1}$ defines a time-change in the sense of \cite[Definition 1.2 p.180]{revuz-yor}. We define the time-changed process
	 \[
	  \tilde Y(t) := X(R^{-1}(t)) \quad t \geq 0,
	 \]
	 which is a Markov process for the filtration $\tilde{\mathcal{G}}_t := \mathcal{F}_{R^{-1}(t)}$ (see \cite[Theorem 2.18 p. 417]{revuz-yor} for a proof of this fact in the case of a Brownian motion and \cite[Theorem 10.10]{D1965} for the general case).
	 Let now $(\tilde L, D(\tilde L))$ be the generator of the semigroup induced by $\tilde Y$ on $C_0(\Gamma)$. We claim that the semigroups generated by $(\tilde L, D(\tilde L))$ and $(L,D(L))$, for $\eta = 0$, coincide. Let $\bar x > 0$ and set $\bar \sfx  =(i, \bar x)$ for some $i \in \{ 1, \dots, N \}$. Let $f \in D(\tilde L)  \cap C^2_b(\Gamma)$ and let $\psi \in C^\infty(\RR_+)$ be such that
	 \[
	  0 \leq \psi \leq 1 \quad \textnormal{and} \quad \psi(x) = \begin{cases}
	                                                             0 \quad \textnormal{for } x \leq \bar x/3, \\
	                                                             1 \quad \textnormal{for } x \geq 2 \bar x /3 ,
	                                                            \end{cases}
	 \]
	and define $g \in C^2_b(\Gamma)$ by setting $g_i(x) = f_i(x) \psi(x)$. Notice that $g \in D(L)$.
	Set also $\tau := \inf \{ t \geq 0 : \tilde y(t) \leq 2 \bar x / 3 \}$ and notice that $R(t) = t$ on $[0,\tau]$ if $X$ is started at $\bar \sfx$. Then, using Dynkin's formula, we have
	 \begin{align*}
	   \mathbf E_{\bar \sfx}[f(\tilde{Y}(t \wedge \tau))] - f(\bar \sfx) & =  \mathbf E_{\bar \sfx}[g(X(R^{-1}(t \wedge \tau))] - g(\bar \sfx) = \mathbf E_{\bar \sfx}[g(X(t \wedge \tau)] - g(\bar \sfx) \\
	   & = \mathbf E_{\bar \sfx} \left [ \int_0^{t \wedge \tau} L g(X(s)) \, ds \right ].
	 \end{align*}
	 It follows that
	 \begin{align*}
	  \tilde Lf(\bar \sfx) & = \lim_{t \to 0^+} t^{-1} \left (\mathbf E_{\bar \sfx}[f(\tilde{Y}(t \wedge \tau))] - f(\bar \sfx) \right ) \\
	  & =   \lim_{t \to 0^+} t^{-1} \mathbf E_{\bar \sfx} \left [ \int_0^{t \wedge \tau} L g(X(s)) \, ds \right ] = Lg(\bar \sfx) = Lf(\bar \sfx),
	 \end{align*}
	 so that $\tilde Lf(\sfx) = L f(\sfx)$ for all $\sfx \neq \sfv$ and $f \in D(\tilde L) \cap C^2_b(\Gamma)$.

	 We now turn to the characterization of $D(\tilde L)$. Let $f \in D(\tilde L) \cap C^2_b(\Gamma)$ and consider, for all $\delta > 0$, the stopping times
	 \[
	  T_X^\delta := \inf \left \{ t \geq 0 : x(t) \geq \delta \right \},
	 \]
	 and $T_{\tilde Y}^\delta$, defined analogously. Following the argument in \cref{tm:time-change}, one may prove that $T_{\tilde Y}^\delta = R(T_X^\delta)$. Using Taylor's formula, we have
	 \begin{align*}\label{eq:reverse_time_change_1}
	 \mathbf{E}_\mathsf{v} \left [f(\tilde Y(T_{\tilde Y}^\delta)) - f(\mathsf{v}) \right ] & = \mathbf{E}_\mathsf{v} \left [f(X(T_{\tilde X}^\delta)) - f(\mathsf{v}) \right ]
	 = \sum_{j=1}^N \rho^\delta_j \left (f_j(\delta) - f_j(0) \right) \\
	 & = \sum_{j=1}^N \rho_j^\delta \left (f_j'(0) \delta + O(\delta^2) \right),
	\end{align*}
	where $\rho_j^\delta = \mathbf{P}_\mathsf{v} \left ( i(T_X^\delta) = j \right)$. On the other hand, Dynkin's formula yields
	\begin{align*}
	 \left \lvert \mathbf{E}_\mathsf{v} \left [f(\tilde Y(T_{\tilde Y}^\delta)) - f(\mathsf{v}) \right ] \right \rvert & = \left \lvert \mathbf{E}_\mathsf{v} \left [\int_0^{T_{\tilde Y}^\delta} L f( \tilde Y(s)) \, ds \right ] \right \rvert = \left \lvert \mathbf{E}_\mathsf{v} \left [\int_0^{R(T_{X}^\delta)} L f( X(R^{-1}(s))) \, ds \right ] \right \rvert \\
	 & \leq C \mathbf E_{\sfv} \left [ R(T_X^\delta) \right] = C  \left (\mathbf E_{\sfv} \left [ T_X^\delta \right ] - \mathbf E_{\sfv} \left [ \int_0^{T_X^\delta} \mathbf 1_{\{X(s) = \sfv \}} \, ds \right ] \right ) \\
	 & = O(\delta^2),
	\end{align*}
	where we used \cref{lem:MFPT,lem:occupation} to deduce that last line.
	Dividing by $\delta$ and sending $\delta \downarrow 0$, using \cref{lem:kirchhoff}, we deduce that
	\[
	 \sum_{j=1}^N \rho_j f_j'(0) = 0,
	\]
	which proves that $D(\tilde L) \cap C^2_b(\Gamma) \subset D(L)$. We conclude that $\tilde Y$ is a nonsticky diffusion on $\Gamma$. According to \cref{tm:time-change} we can define a sticky process $\tilde X(t) = \tilde Y(V^{-1}(t)) = X(R^{-1}(V^{-1}(t)))$, with stickiness parameter $\eta$. We claim that $R^{-1}(V^{-1}(t)) = t$ $\mathbf P_\sfv$-almost surely. Since
	\[
	 t = R^{-1} \circ V^{-1} \circ V \circ R(t) \quad \textnormal{almost surely,}
	\]
	it is enough to prove that $V(R(t)) = t$ $\mathbf P_\sfv$-almost surely. We a.s. have
	\begin{align*}
	 \frac{1}{\delta} \int_0^{R(t)} \mathbf 1_{\{\tilde y(s) \leq \delta \}} \, ds & = \frac{1}{\delta} \int_0^{t} \mathbf 1_{\{x(s) \leq \delta \}} dR(s) = \frac{1}{\delta} \int_0^{t} \mathbf 1_{\{x(s) \leq \delta \}} \, ds - \frac{1}{\delta} \int_0^t \mathbf 1_{\{x(s) = 0 \}} \, ds \\
	 & = \frac{1}{\delta} \int_0^{t} \mathbf 1_{\{ 0< x(s) \leq \delta \}} \, ds
	\end{align*}
	so that, using \cref{tm:FW_FS}, we have
	\begin{equation}\label{eq:characterization_1}
	 \ell_{\tilde Y}(R(t)) = \lim_{\delta \to 0} \frac{\Lambda}{\delta} \int_0^{t} \mathbf 1_{\{ 0< x(s) \leq \delta \}} \, ds \quad \textnormal{in expectation},
	\end{equation}
	where
	\[
	 \Lambda := \left( \sum_{i=1}^N \frac{2 \rho_i}{\sigma_i(0)^2} \right)^{-1}.
	\]
	Recall that
	\[
	 V(R(t)) = R(t) + \eta \ell_{\tilde Y}(R(t)) = t - \int_0^t \mathbf \mathbf 1_{\{x(s) = 0 \}} \, ds + \eta \ell_{\tilde Y}(R(t)).
	\]
	Therefore, we obtain the claim if we are able to prove
	\[
	 \int_0^t \mathbf \mathbf 1_{\{x(s) = 0 \}} \, ds = \lim_{\delta \to 0}  \frac{\eta \Lambda}{\delta} \int_0^{t} \mathbf 1_{\{ 0< x(s) \leq \delta \}} \, ds
	\]
	 $\mathbf P_\sfv$-almost surely as a consequence of \eqref{eq:characterization_1}.
	Recall that we have defined $\ell_{\tilde X}(t) := \ell_{\tilde Y}(V^{-1}(t))$ and that
	\begin{equation}\label{eq:characterization_V}
		V^{-1}(t) = t - \eta \ell_{\tilde X}(t) = t - \int_0^t \mathbf 1_{\{\tilde x(s) = 0 \}} \, ds,
	\end{equation}
	where we have used \cref{prop:occupation}. Noticing also that, using \eqref{eq:characterization_V},
	\begin{align*}
	 \int_0^{V^{-1}(t)} \mathbf{1}_{\{\tilde y(s) \leq \delta \}}\, ds & = \int_0^t \mathbf 1_{\{\tilde x(s) \leq \delta \}} d V^{-1}(s) = \int_0^t \mathbf 1_{\{0 < \tilde x(s) \leq \delta \}} d s,
	\end{align*}
	we deduce from \cref{tm:FW_FS} that
	\begin{equation}\label{eq:characterization_time_changed_local_times}
	 \ell_{\tilde X}(t) = \lim_{\delta \to 0} \frac{\Lambda}{\delta } \int_0^t \mathbf 1_{\{0 < \tilde x(s) \leq \delta \}} d s \quad \textnormal{in expectation}.
	\end{equation}
	Using \cref{prop:occupation} and the fact that $X$ and $\tilde X$ have the same transition function, we deduce that
	\begin{align*}
	 & \mathbf E_{\sfx} \left [ \left \lvert \int_0^t \mathbf 1_{\{x(s) = 0 \}} \, ds - \frac{\eta \Lambda}{\delta} \int_0^t \mathbf 1_{\{ 0< x(s) \leq \delta \}} \, ds \right \rvert \right ] \\
	 & \qquad = \mathbf E_{\sfx} \left [ \left \lvert \int_0^t \mathbf 1_{\{\tilde x(s) = 0 \}} \, ds - \frac{\eta \Lambda}{\delta} \int_0^t \mathbf 1_{\{ 0< \tilde x(s) \leq \delta \}} \, ds \right \rvert \right ] \\
	 & \qquad = \eta \mathbf E_{\sfx} \left [ \left \lvert \ell_{\tilde X}(t) - \frac{ \Lambda}{\delta} \int_0^t \mathbf 1_{\{ 0< \tilde x(s) \leq \delta \}} \, ds \right \rvert \right ] \xrightarrow{\delta \to 0} 0.
	\end{align*}
	The claim is proved. In particular, it follows that $X(t) = \tilde X(t)$ $\mathbf P_\sfx$-almost surely for every $t \geq 0$.
	Since these processes are continuous, they are in fact indistinguishable. We now check that $\tilde X$ is adapted to $(\widehat {\mathcal F}_t^X)_{t \geq 0}$, the completion  of the natural filtration $(\mathcal F_t^X)_{t \geq 0}$ of $X$ with respect to $\left ( \mathbf P_\sfx \right)_{\sfx \in \Gamma}$. Indeed, there is a $\mathbf P_\sfx$-negligible set $N$ such that $X(t,\omega) = \tilde X(t,\omega)$ for every $\omega \notin N$. Then, for $E \in \mathcal{F}_t^X$ we have
	\[
	 \{ \omega : \tilde X(t,\omega) \in E \} = \left (\{ \omega : X(t,\omega) \in E \} \cap N^{\mathsf{c}} \right ) \cup \left ( \{ \omega : \tilde X(t,\omega) \in E \} \cap N \right )\in \widehat {\mathcal F}_t^X.
	\]
	Moreover, from \eqref{eq:characterization_time_changed_local_times}, we deduce that $\ell_{\tilde X}$ is also adapted to $(\widehat{\mathcal F}_t^X)_{t \geq 0}$.
	This concludes the proof if we set $\ell_X = \ell_{\tilde X}$.

	\end{proof}

	\subsection{Stochastic differential equation}

	Applying \cref{tm:FW_FS}-(ii) to the nonsticky process $Y(t) = (j(t),y(t))$  we have that
	\begin{align}
	\label{sde:no-sticky}
	dy(t) = \sigma_{j(t)}(y(t)) dW(t) + b_{j(t)}(y(t)) dt + d\ell_Y(t),
	\end{align}
	where $W$ is a one-dimensional Brownian motion measurable with adapted to $\cG_t$, the natural filtration generated by $Y$. On the other hand, an SDE representation for sticky process was obtained in \cite[Theorem 3.2]{salins} in the particular case where the number of edges is $N=2$. The point of the next result is to extend this representation to the case $N > 2$.
	
	\begin{tm}\label{tm:SDE}
		Up to an extension of the filtered probability space, there exists a one dimensional Brownian motion $\widetilde W$, such that the sticky process $ X(t) = (x(t),i(t))$ satisfies
		\begin{align}
		\label{sde:sticky}
			dx(t) = \sigma_{i(t)}(x(t)) \mathbf{1}_{\{x(t) \ne 0\}}\, d\widetilde W(t) + b_{i(t)}(x(t)) \mathbf{1}_{\{x(t) \ne 0\}}\, dt + d\ell_X(t),
	\end{align}
	where $\ell_X$ is the process obtained in \cref{tm:characterization}.
		
	\end{tm}
\begin{proof}
	The only difference with \cite[Theorem 3.2]{salins}, is that the existence of the SDE for the diffusion $y$ is given by \eqref{sde:no-sticky}. For the rest, we can proceed with the same steps.  Using \cref{tm:characterization}, we may assume that, almost surely, $X(t) = \tilde Y(V^{-1}(t))$ and $\ell_X(t) = \ell_{\tilde Y}(V^{-1}(t))$ for some nonsticky diffusion $\tilde Y(\cdot) = ( j(\cdot), \tilde y (\cdot))$. From \eqref{eq:time_change_representation} and \eqref{simple-case}, we have
		\begin{align*}
		V^{-1}(t)=t - \eta \ell_X(t) = \int_0^t \mathbf{1}_{\{X(s) \ne \sfv\}}\, ds
		\end{align*}
		Using the fact that $x(t)= \tilde y(V^{-1}(t))$, we have
		\begin{align*}
		\int_0^t \mathbf{1}_{\{X(s) \ne \sfv\}}\, ds= \int_0^t \mathbf{1}_{\{x(s) \ne 0\}}\, ds.
		\end{align*}
		By applying the time change $V^{-1}(t)$, we get
		\begin{align*}
		dx(t) = \sigma(x(t)) d S(t)+ b(x(t)) \mathbf{1}_{\{x(t) \ne 0\}}\, dt + d\ell_X(t)
		\end{align*}
		where $S(t)= W\left(\int_0^t \mathbf{1}_{\{x(s) \ne 0\}}\, ds\right) $. We define
		\begin{align*}
		\widetilde{W}(t):=W\left(\int_0^t \mathbf{1}_{\{x(s) \ne 0\}}\, ds\right) +W_2\left(\int_0^t \mathbf{1}_{\{x(s) = 0\}}\, ds\right),
		\end{align*}
		where $W_2$ is a Brownian motion independent of $W$ (which is known to exist up to extension of the probability space, see \cite[Remark 4.1 p.169]{KS1991}).
		This concludes the proof since $d S(t)=\mathbf{1}_{\{x(t) \ne 0\}}\, d\widetilde{W}(t) $.
\end{proof}

	Let us now provide a result, consistent with the well-known convergence of the upcrossing number to the local time, which we will later use in the proof of It\^o's formula.
	\begin{coro}\label{coro:local_time_approx}
		Let $f \colon \RR_+ \to \RR$ be bounded and continuous and let $X(t) = (i(t),x(t))$ be a sticky process on $\Gamma$ with generator \eqref{generator}-\eqref{eq:domain} and stickiness parameter $\eta$, started at $\sfv$. Define, for $\delta > 0$, the collections of stopping times $(\tau_n^\delta)_{n \in \NN}$ and $(\theta_n^\delta)_{n \in \NN}$ by $\tau_0^\delta = \theta_0^\delta = 0$ and
		\begin{equation}\label{eq:FS_stopping_times}
		\begin{split}
			\theta_n^\delta = \inf \left \{ t \geq \tau_{n-1}^\delta : x(t) = \delta \right \},  \\
			\tau_n^\delta = \inf \left \{ t \geq \theta_{n}^\delta : x(t) = 0 \right \},
		\end{split}
		\end{equation}
		for $n \geq 1$. Then, for every $t \geq 0$,
		\[
		 \lim_{\delta \to 0} \delta \sum_{n \in \NN} f(\theta_{n+1}^\delta) \mathbf{1}_{\{\theta_{n+1}^\delta \leq t \}} = \int_0^t f(s) d \ell_X(s) \quad \textnormal{in probability}.
		\]
		In particular
		\[
		 \lim_{\delta \to 0} \delta \sum_{n \in \NN} \mathbf{1}_{\{\theta_{n+1}^\delta \leq t \}} = \ell_X(t) \quad \textnormal{in probability.}
		\]
	\end{coro}

	\begin{proof}
		The argument is reminiscent of the proof for the approximation of local times of semimartigales by the number of upcrossings \cite[Proposition 9.11, page 248]{LG2016}. Notice first that, for every $n \in \NN$, we have, from \cref{tm:SDE},
		\begin{align*}
		R_n := f(\theta_{n+1}^\delta) (x(\theta_{n+1}^\delta \wedge t) - x(\tau_n^\delta \wedge t))=
	f(\theta_{n+1}^\delta) \int_{\tau_{n}^\delta\wedge t}^{\theta_{n+1}^\delta\wedge t} \sigma_{i(s)}(x(s)) \mathbf{1}_{\{ x(s) > 0 \}} \,  d \widetilde W(s) \\
				 + f(\theta_{n+1}^\delta) \int_{\tau_{n}^\delta\wedge t}^{\theta_{n+1}^\delta\wedge t} b_{i(s)}(x(s)) \mathbf{1}_{\{ x(s) > 0 \}} \, ds + f(\theta_{n+1}^\delta)( \ell_X(\theta_{n+1}^\delta\wedge t) -  \ell_X(\tau_n^\delta\wedge t)).
		\end{align*}
		Summing over $n \in \NN$ we obtain
		\begin{align*}
		\sum_{n \in \NN} R_n & = \int_{0}^t \sum_{n \in \NN} f(\theta_{n+1}^\delta) \mathbf{1}_{(\tau_n^\delta, \theta_{n+1}^\delta]}(s) \sigma_{i(s)}(x(s)) \mathbf{1}_{\{ x(s) > 0 \}} \,  d \widetilde W(s)\\
		& \quad +\int_{0}^t \sum_{n \in \NN} f(\theta_{n+1}^\delta) \mathbf{1}_{(\tau_n^\delta, \theta_{n+1}^\delta]}(s) b_{i(s)}(x(s)) \mathbf{1}_{\{ x(s) > 0 \}} \, ds \\
		& \quad +\sum_{n \in \NN}f(\theta_{n+1}^\delta)( \ell_X(\theta_{n+1}^\delta\wedge t) -  \ell_X(\tau_n^\delta\wedge t))
		\end{align*}
		where we have used the fact that there are only finitely many values of $n$ such that $\theta_{n+1}^\delta \leq t$. Since $b$ and $f$ are bounded and
		\[
		 0 \leq  \mathbf{1}_{(\tau_n^\delta, \theta_{n+1}^\delta]}(s) \mathbf{1}_{\{ x(s) > 0 \}} \leq \mathbf{1}_{\{ 0 < x(s) \leq \delta \} } \xrightarrow{\delta \to 0}{} 0 \quad \textnormal{almost surely,}
		\]
		we have
		\begin{align*}
		\int_{0}^t \sum_{n \in \NN} f(\theta_{n+1}^\delta) \mathbf{1}_{(\tau_n^\delta, \theta_{n+1}^\delta]}(s) b_{i(s)}(x(s)) \mathbf{1}_{\{ x(s) > 0 \}} \, ds \xrightarrow{\delta \downarrow 0} 0 \quad \text{almost surely.}
		\end{align*}
		Similarly, using \cite[Proposition 5.8, page 111]{LG2016}, we get
		\begin{align*}
		\int_{0}^t \sum_{n \in \NN} f(\theta_{n+1}^\delta) \mathbf{1}_{(\tau_n^\delta, \theta_{n+1}^\delta]}(s) \sigma_{i(s)}(x(s)) \mathbf{1}_{\{ x(s) > 0 \}} \,  d \widetilde W(s) \xrightarrow{\delta \downarrow 0} 0 \quad \text{in probability,}
		\end{align*}
		From the properties of Riemann-Stieltjes integrals and the fact that $\ell_X$ a.s. does not increase between $\theta_{n+1}^\delta$ and $\tau_{n+1}^\delta$, we also have
		\begin{align*}
		 \int_0^t f(s) d\ell_X(s) & = \lim_{\delta \to 0 } \bigg[ \sum_{n \in \NN} f(\theta_{n+1}^\delta) \left( \ell_X(\theta_{n+1}^\delta\wedge t) -  \ell_X(\tau_n^\delta \wedge t) \right ) \\
		 & \quad + \sum_{n \in \NN} f(\tau_{n+1}^\delta) \left( \ell_X(\tau_{n+1}^\delta\wedge t) -  \ell_X(\theta_{n+1}^\delta \wedge t) \right) \bigg ] \\
		 & = \lim_{\delta \to 0 } \sum_{n \in \NN} f(\theta_{n+1}^\delta) \left( \ell_X(\theta_{n+1}^\delta\wedge t) -  \ell_X(\tau_n^\delta \wedge t) \right )
		\end{align*}
		almost surely.
				We conclude the proof by observing that
				\begin{align*}
					\sum_{n \in \NN} f(\theta_{n+1}^\delta) (x(\theta_{n+1}^\delta \wedge t) - x(\tau_n^\delta \wedge t)) = \delta  \sum_{n \in \NN} f(\theta_{n+1}^\delta) \mathbf{1}_{\{\theta_{n+1}^\delta \leq t \}} + O(\delta),
				\end{align*}
			Indeed, $x(\theta_{n+1}^\delta \wedge t) - x(\tau_n^\delta \wedge t) = \delta$ when $\theta_{n+1}^\delta \leq t$, and when $\theta_{n+1}^\delta > t$, we have $x(t) - x(\tau_n^\delta \wedge t) \leq \delta$. Hence, the claim holds.
	\end{proof}

\subsection{Freidlin-Sheu-It\^o formula}

The goal of this section is to extend the It\^o type formula proved in \cite[Lemma 2.3]{FreidlinSheu} for the nonsticky diffusion to the case of sticky diffusions. Note that the special case $N=2$ was considered in \cite[Lemma 3.5]{salins}.

\begin{tm}\label{tm:Ito}
	Let $X(t) = (i(t), x(t))$ be the process generated by \eqref{generator2} and \eqref{eq:domain} with $\eta > 0$, started at $X(0) = \mathsf{x} \in \Gamma$, and let $f \in C_{b}^{1,2}(\RR_+ \times \Gamma)$. Then,
	\begin{equation}\label{eq:Ito_1}
	\begin{split}
	& f(t,X(t)) = f(0,\mathsf{x}) + \int_0^t \left ( \partial_s f(s,X(s)) +  Lf(s,X(s)) \right ) \mathbf 1_{\{ X(s) \neq \sfv \}}  ds \\
	& \, + \int_0^t \sigma(X(s)) \partial_x f(s,X(s)) \mathbf{1}_{\{X(s) \neq \sfv \}} d \widetilde W(s)
	 + \int_0^t \left ( \eta\, \partial_s f(s,\sfv) +  \sum_{i=1}^N \rho_i \partial_x f_i(s,0) \right ) d \ell_X(s),
	\end{split}
	\end{equation}
	where $\widetilde W$ is the Brownian motion obtained in Theorem \ref{tm:SDE}. In particular, if $Lf(t,\cdot) \in C(\Gamma)$ for all $t \geq 0$, we have
	\begin{equation}\label{eq:Ito_2}
	\begin{split}
	& f(t,X(t)) = f(0,\mathsf{x}) + \int_0^t \left ( \partial_s f(s,X(s)) +  Lf(s,X(s)) \right )  ds \\
	& \, + \int_0^t \sigma(X(s)) \partial_x f(s,X(s)) \mathbf{1}_{\{X(s) \neq \sfv \}} d \widetilde W(s) + \int_0^t \left (\sum_{i=1}^N \rho_i \partial_x f_i(s,0) - \eta\, L f(s,\sfv) \right ) d \ell_X(s).
	\end{split}
	\end{equation}
\end{tm}
\begin{proof}
	We follow the proof of the nonsticky case (see \cite{FreidlinSheu,H2011}).

Without loss of generality, we may assume that $\sfx=(i,0)$. Fix $\delta > 0$. We recursively define the following sequence of stopping times : $\tau_0^\delta = \theta_0^\delta = 0$ and, for every integer $n \geq 1$,
\begin{equation}
\begin{split}
\theta_n^\delta = \inf \left \{ t \geq \tau_{n-1}^\delta : x(t) = \delta \right \},  \\
\tau_n^\delta = \inf \left \{ t \geq \theta_{n}^\delta : x(t) = 0 \right \}.
\end{split}
\end{equation}
In the case where $\sfx \neq (i, 0)$, we define instead
\[
 \tau_0^\delta = \inf \left \{ t \geq 0 : x(t) = 0 \right \},
\]
and the rest of the proof is similar.
 From the SDE \eqref{sde:sticky}, we have that the number of upcrossings of $x$ from $0$ to $\delta$ up to time $t$ is almost surely finite. We then decompose
\begin{align*}
f(t,X(t)) - f(0,\mathsf{x}) &  = \sum_{n=0}^{+\infty} f(\theta_{n+1}^{\delta} \wedge t, X(\theta_{n+1}^{\delta} \wedge t)) - f(\theta_{n}^{\delta} \wedge t, X(\theta_{n}^{\delta} \wedge t)) \\
& = Q_1^\delta + Q_2^\delta + Q_3^\delta,
\end{align*}
where
\begin{equation}
\begin{split}
Q_1^\delta & := \sum_{n \in \NN} \left (f(\theta_{n+1}^{\delta} \wedge t, X(\theta_{n+1}^{\delta} \wedge t)) - f(\tau_{n}^{\delta} \wedge t, X(\tau_{n}^{\delta} \wedge t)) \right)  \\ & \quad - \sum_{j=1}^N \sum_{n \in \NN} \delta \partial_x f_j(\theta_{n+1}^{\delta},0) \mathbf{1}_{\{\theta_{n+1}^{\delta} \leq t,\, i(\theta_{n+1}^{\delta}) = j\}},
\end{split}
\end{equation}
\begin{equation}
Q_2^\delta := \sum_{j=1}^N \sum_{n \in \NN} \delta \partial_x f_j(\theta_{n+1}^{\delta},0) \mathbf{1}_{\{\theta_{n+1}^{\delta} \leq t,\, i(\theta_{n+1}^{\delta}) = j\}},
\end{equation}
and
\begin{equation}
Q_3^\delta := \sum_{n \in \NN} f(\tau_{n}^{\delta} \wedge t, X(\tau_{n}^{\delta} \wedge t)) - f(\theta_{n}^{\delta} \wedge t, X(\theta_{n}^{\delta} \wedge t)).
\end{equation}
	We also write $Q_1^\delta = Q_{(1,1)}^\delta + Q_{(1,2)}^\delta$, where
\begin{equation}
\begin{split}
Q_{(1,1)}^\delta & := \sum_{j=1}^N \sum_{n \in \NN} \left (f_j(\theta_{n+1}^{\delta}, x(\theta_{n+1}^{\delta})) - f_j(\tau_{n}^{\delta}, x(\tau_{n}^{\delta})) - \delta \partial_x f_j(\theta_{n+1}^{\delta},0) \right ) \mathbf{1}_{\{\theta_{n+1}^{\delta} \leq t,\, i(\theta_{n+1}^{\delta}) = j\}}
\end{split}
\end{equation}
and
\begin{equation}
Q_{(1,2)}^\delta := \sum_{j=1}^N \sum_{n \in \NN} \left (f_j(\theta_{n+1}^{\delta} \wedge t, x(\theta_{n+1}^{\delta} \wedge t)) - f_j(\tau_{n}^{\delta} \wedge t, x(\tau_{n}^{\delta} \wedge t)) \right) \mathbf{1}_{\{\theta_{n+1}^{\delta} > t,\, i(\theta_{n+1}^{\delta}) = j\}}.
\end{equation}

\noindent\underline{Step 1}: $Q_{1}^\delta \xrightarrow{\delta \to 0}{} \eta \int_0^t \partial_s f(s,\sfv) \ell_X(ds)$ in probability.

From the definition of $\theta_{n} , \, \tau_{n}$ and the continuity of $X$ we know that, for every $n \geq 0$ and $j \in \{1, \dots, N\}$,
\[
f_j(\theta_{n+1}^\delta, x(\theta_{n+1}^\delta)) = f_j(\theta_{n+1}^\delta, \delta) \quad \textnormal{and} \quad f_j(\tau_{n}^\delta, x(\tau_{n}^\delta)) = f_j(\tau_{n}^\delta,0) \quad \textnormal{almost surely}.
\]
It then follows that, on each edge $j$,
\begin{align*}
& f_j(\theta_{n+1}^\delta, x(\theta_{n+1}^\delta)) - f_j(\tau_{n}^\delta, x(\tau_{n}^\delta)) - \delta \partial_x f_j(\theta_{n+1}^{\delta},0) \\
& \quad = f_j(\theta_{n+1}^\delta, \delta) - f_j(\tau_{n}^\delta, 0) - \delta \partial_x f_j(\theta_{n+1}^{\delta},0) \\
& \quad = f_j(\theta_{n+1}^\delta, \delta) - f_j(\theta_{n+1}^\delta, 0) - \delta \partial_x f_j(\theta_{n+1}^{\delta},0) + f_j(\theta_{n+1}^\delta, 0) - f_j(\tau_{n}^\delta, 0) \\
& \quad = \int_0^\delta \partial_z f_j(\theta_{n+1}^\delta,z) - \partial_x f_j(\theta_{n+1}^\delta, 0)\, dz + \int_{\tau_n^\delta}^{\theta_{n+1}^\delta} \partial_s f_j(s,0)\, ds.
\end{align*}
Therefore, we have
\begin{align*}
& \left \lvert Q_{(1,1)}^\delta - \left [ \sum_{j=1}^N \sum_{n \in \NN} \int_{\tau_n^\delta}^{\theta_{n+1}^\delta} \partial_s f_j(s,0) \mathbf{1}_{\{\theta_{n+1}^{\delta} \leq t,\, i(\theta_{n+1}^{\delta}) = j\}}  \, ds \right ] \right \rvert \\ & \qquad   = \left \lvert Q_{(1,1)}^\delta - \sum_{n \in \NN} \int_0^{\theta_{n+1}^\delta } \partial_s f(s,0) \mathbf{1}_{\{ x(s) \leq \delta,\, \theta_{n+2} > t  \}} \, ds \right \rvert \\
& \qquad \leq N \delta^2 \lVert \partial_x^2 f \rVert_{\infty} \sum_{n \in \NN} \mathbf{1}_{\{\theta_{n+1}^{\delta} \leq t\}}.
\end{align*}

On the other hand
\begin{align*}
& \left \lvert Q_{(1,2)}^\delta - \sum_{j=1}^N \sum_{n \in \NN} \int_{\tau_n^\delta}^t \partial_s f_j(s,0) \mathbf{1}_{\{\theta_{n+1}^{\delta} > t,\, i(\theta_{n+1}^{\delta}) = j\}}\, ds \right \rvert \\
& \quad  = \left \lvert \sum_{j=1}^N \sum_{n \in \NN} \left (f_j(t, x(t)) - f_j(\tau_{n}^{\delta}, x(\tau_{n}^{\delta})) - \int_{\tau_n^\delta}^t \partial_s f_j(s,0)\, ds \right) \mathbf{1}_{\{\theta_{n+1}^{\delta} > t,\, i(\theta_{n+1}^{\delta}) = j\}} \right \rvert \\
& \quad\leq \sum_{j=1}^N \sum_{n \in \NN} \left \lvert (f_j(t, x(t)) - f_j(\tau_{n}^{\delta}, x(\tau_{n}^{\delta})) - \int_{\tau_n^\delta}^t \partial_s f_j(s,0)\, ds \right \rvert \mathbf{1}_{\{\theta_{n+1}^{\delta} > t,\, i(\theta_{n+1}^{\delta}) = j\}} \\
& \quad  \leq \delta \lVert\partial_x f \rVert_{\infty}.
\end{align*}
We conclude that
\begin{align*}
& \left \lvert Q_1^\delta - \sum_{n \in \NN} \int_{\tau_{n}^\delta \wedge t}^{\theta_{n+1}^\delta \wedge t} \partial_s f(s,0) \mathbf{1}_{\{x(s) \leq \delta \}}\, ds \right \rvert \\
& \quad \leq \left \lvert Q_{(1,1)}^\delta - \sum_{j=1}^N \sum_{n \in \NN} \int_{\tau_n^\delta}^{\theta_{n+1}^\delta} \partial_s f_j(s,0) \mathbf{1}_{\{\theta_{n+1}^{\delta} \leq t,\, i(\theta_{n+1}^{\delta}) = j\}}\, ds  \right \rvert \\ & \qquad + \left \lvert Q_{(1,2)}^\delta - \sum_{j=1}^N \sum_{n \in \NN} \int_{\tau_n^\delta}^t \partial_s f_j(s,0) \mathbf{1}_{\{\theta_{n+1}^{\delta} > t,\, i(\theta_{n+1}^{\delta}) = j\}}\, ds \right \rvert \\
& \quad \leq N \delta^2 \lVert \partial_x^2 f \rVert_{\infty} \sum_{n \in \NN} \mathbf{1}_{\{\theta_{n+1}^{\delta} \leq t\}} + \delta \lVert\partial_x f \rVert_{\infty}.
\end{align*}
From \cref{coro:local_time_approx}, we recall that $\delta \sum_{n \in \NN} \mathbf{1}_{\{\theta_{n+1}^{\delta} \leq t\}}$ goes, for $\delta \downarrow 0$, in probability to $\ell_X(t)$. Then, we have
\begin{align*}
 N \delta^2 \lVert \partial_x^2 f \rVert_{\infty} \sum_{n \in \NN} \mathbf{1}_{\{\theta_{n+1}^{\delta} \leq t\}} \xrightarrow{\delta \to 0}{} 0 \quad \text{in probability}
\end{align*}
and 
\begin{align*}
\delta \lVert\partial_x f \rVert_{\infty} \xrightarrow{\delta \to 0}{}.
\end{align*}
Therefore,
\begin{align*}
Q_1^\delta - \sum_{n \in \NN} \int_{\tau_{n}^\delta \wedge t}^{\theta_{n+1}^\delta \wedge t} \partial_s f(s,0) \mathbf{1}_{\{X(s) \leq \delta \}}\, ds \xrightarrow{\delta \to 0}{} 0 \quad \text{in probability.}
\end{align*}
But, from Lebesgue's dominated convergence theorem, we get
\[
\lim_{\delta \to 0} \sum_{n \in \NN} \int_{\tau_{n}^\delta \wedge t}^{\theta_{n+1}^\delta \wedge t} \partial_s f(s,0) \mathbf{1}_{\{X(s) \leq \delta \}}\, ds = \int_0^t \partial_s f(s,0) \mathbf{1}_{\{ X(s) = \sfv \}}\, ds,
\]
Then, our first claim holds true since \cref{prop:occupation} yields
\[
\int_0^t \partial_s f(s,0) \mathbf{1}_{\{ X(s) = \sfv \}}\, ds = \eta \int_0^t \partial_s f(s,0) \ell_X(ds).
\]
	\noindent\underline{Step 2}: $Q_{2}^\delta \xrightarrow{\delta \to 0}{} \int_0^t \sum_{i=1}^N  \rho_i \partial_x f_i(s,0) d\ell_X(s)$ in probability.

	Notice first that since the process is stated from $\sfv$, it follows from the SDE \eqref{sde:sticky}, the occupation time formula \eqref{simple-case} and the fact that $\sfv$ is not a trap for $X$ (\cref{prop:trap}) that
	\begin{equation}\label{eq:local_time_positive}
	 \ell_X(t) > 0 \quad \textnormal{a.s. for all $t > 0$.}
	\end{equation}
	Write
	\[
	Q_2^\delta = \sum_{n \in \NN} \delta  \mathbf{1}_{\{\theta_{n+1}^{\delta} \leq t \}} \sum_{j=1}^N \partial_x f_j(\theta_{n+1}^{\delta},0) \mathbf{1}_{\{i(\theta_{n+1}^{\delta}) = j\}}.
	\]
	 Define
	\[
		\tilde \rho_j^\delta = \mathbf{P}_\sfv \left (i(X (\theta_{n}^\delta)) = j \right )
	\]
	 recall from \cref{lem:kirchhoff} that
	\begin{align}
	\label{edge-selection}
	\lim_{\delta \to 0} \tilde \rho_j^\delta =\rho_j.
	\end{align}
	Decompose
	\begin{align*}
	& \left \lvert Q_2^\delta - \sum_{i=1}^N \int_0^t \rho_i \partial_x f_i(s,0) d\ell_X(s) \right \rvert \\
	& \quad \leq \left \lvert \sum_{j=1}^N \sum_{n \in \NN} \left(\mathbf{1}_{\{i(X(\theta_{n+1}^\delta)) = j\}} - \tilde \rho_j^\delta \right) \delta \partial_xf_j(\theta_{n+1}^\delta,0) \mathbf{1}_{\{\theta_{n+1}^\delta \leq t \}} \right \rvert \\
	& \qquad + \left \lvert \sum_{j=1}^N \tilde \rho_j^\delta \left( \sum_{n \in \NN} \delta \partial_xf_j(\theta_{n+1}^\delta,0) \mathbf{1}_{\{\theta_{n+1}^\delta \leq t \}} - \int_0^t \partial_x f_j(s,0)\, d \ell_X(s) \right) \right \rvert \\
	& \qquad + \left \lvert \sum_{j=1}^N \left( \tilde \rho_j^\delta - \rho_j \right) \int_0^t \partial_x f_j(s,0) \, d\ell_X(s) \right \rvert \\
	& \quad =: R_1^\delta + R_2^\delta + R_3^\delta.
	\end{align*}
	Notice that it directly follows from \eqref{edge-selection} that $R_3^\delta \to 0$ in probability.
Moreover, from \cref{coro:local_time_approx}, we have
\begin{align}\label{eq:CV_ST}
\sum_{n \in \NN} \delta  \mathbf{1}_{\{\theta_{n+1}^{\delta} \leq t \}}  \partial_x f_j(\theta_{n+1}^{\delta},0) \xrightarrow{\delta \to 0}{} \int_0^t \partial_x f_j(s,0) d\ell_X(s) \quad \text{in probability}
\end{align}
so that $R_2^\delta \to 0$ in probability.
We are therefore left with the case of $R_1^\delta$.

We set $N_\delta(t) = \sum_{n \in \NN}  \mathbf{1}_{\{\theta_{n+1}^{\delta} \leq t \}}$ and recall from \cref{coro:local_time_approx} that
\[
	\lim_{\delta \to 0} \delta N_\delta(t) = \ell_X(t) \quad \textnormal{in probability.}
\]
As a consequence of \eqref{eq:local_time_positive}, we deduce that
\begin{equation}\label{eq:CV_Ndelta}
 N_\delta(t) \xrightarrow{\delta \to 0} +\infty \quad \textnormal{in probability.}
\end{equation}
Moreover, from the strong Markov property we know that the random variables $\left(i(X(\theta_{n+1}^\delta)) \right)_{n \in \NN}$ are i.i.d. so that the random variables
\[
 Z_j^{n,\delta} := \mathbf{1}_{\{i(\theta_{n+1}^\delta) = j \}} - \tilde \rho_j^\delta
\]
are also i.i.d. for $j$ and $\delta$ fixed and satisfy
\begin{equation}
 \mathbf{E}_\sfv \left [ Z_j^{n,\delta} \right] = 0 \quad \textnormal{and} \quad \mathbf{E}_\sfv \left[ \left \lvert Z_j^{n,\delta} \right \rvert^{2} \right ] \leq 1.
\end{equation}
Set
\[
 S_j^\delta(m) := \sum_{n=0}^{m-1} Z_j^{n,\delta} \quad \textnormal{for all } m \in \NN
\]
and observe that
\begin{equation}
 \mathbf{E}_\sfv \left[ \left \lvert m^{-1} S_j^{\delta}(m) \right \rvert^2 \right ] \leq \frac{1}{m} \quad \textnormal{for all } m \geq 1.
\end{equation}
Let $\epsilon > 0$. From Tchebychev's  inequality we deduce
\begin{equation}
 \mathbf{P}_{\sfv} \left( \left \lvert m^{-1} S_j^\delta(m) \right \rvert \geq \epsilon \right ) \leq \frac{1}{\epsilon^2} \; \mathbf{E}_\sfv \left[ \left \lvert m^{-1} S_j^{\delta}(m) \right \rvert^2 \right ]  \leq \frac{1}{\epsilon^2 m} \quad \textnormal{for all } m \geq 1.
\end{equation}
In particular, for every $\kappa \in (0,1)$, there exists $M = M(\epsilon,\kappa)> 0$ such that
\begin{equation}
 \mathbf{P}_{\sfv}\left( \left \lvert m^{-1} S_j^\delta(m) \right \rvert \geq \epsilon \right ) \leq \frac{\kappa}{2} \quad \textnormal{for all } m \geq M.
\end{equation}
In addition, it follows from \eqref{eq:CV_Ndelta} that there exists $\delta_0 > 0$ such that
\[
 \mathbf{P}_{\sfv} \left( N_\delta(t) < M \right) \leq \frac{\kappa}{2} \quad \textnormal{for all } 0 < \delta \leq \delta_0.
\]
In the end, we have proved that for every $\epsilon > 0$ and $\kappa \in (0,1)$, there exists $\delta_0 > 0$ such that
\begin{equation}
 \mathbf{P}_{\sfv} \left( \left \lvert N_\delta^{-1}(t) S_j^\delta\left(N_\delta(t) \right) \right \rvert \geq \epsilon \right ) \leq \kappa \quad \textnormal{for all } 0 < \delta \leq \delta_0,
\end{equation}
so that $N_\delta^{-1}(t) S_j^\delta \left(N_\delta(t) \right) \xrightarrow{\delta \to 0}{} 0$ in probability.
Noticing that
\begin{align*}
	R_1^\delta & \leq \delta N_\delta(t) \rVert \partial_x f \rVert_{\infty} \sum_{j=1}^N \left \lvert \tilde \rho_j^\delta  - N_\delta^{-1}(t) \sum_{n=0}^{N_\delta(t)-1} \mathbf{1}_{\{i(\theta_{n+1}^{\delta}) = j\}} \right \rvert \\
	& = \delta N_\delta(t) \rVert \partial_x f \rVert_{\infty} \sum_{j=1}^N \left \lvert N_\delta^{-1}(t) S_j^\delta \left(N_\delta(t)\right) \right \rvert,
\end{align*}
and since
\[
	\mathbf{P}_\sfv \left ( \left \{ \sup_{\delta > 0} \delta N_\delta(t) < \infty \right \} \right ) = 1
\]
because $\ell_X(t)$ is a.s. finite, we conclude that
\[
 \lim_{\delta \to 0} R_1^\delta = 0 \quad \textnormal{in probability}.
\]
This proves the claim.

\noindent\underline{Step 3}:
\begin{align*}
& \lim_{\delta \to 0} Q_{3}^\delta = f(0,x) + \int_0^t \left (\partial_t f(s,X(s)) + Lf(s,X(s)) \right ) \mathbf{1}_{\{X(s) \neq \sfv \}} \\ & \qquad + \int_0^t \sigma(X(s)) \partial_x f(s,X(s)) \mathbf{1}_{\{X(s) \neq \sfv \}} d \widetilde W(s)
\end{align*}
in probability.

This follows from the standard one-dimension It\^o formula and the dominated convergence theorem for stochastic integrals.

\noindent\underline{Step 4}: Conclusion.

It should be noted that
\begin{align*}
f(t,X(t)) - f(0,\mathsf{x}) = Q_1^\delta + Q_2^\delta + Q_3^\delta,
\end{align*}
by using steps 1 to 3, we get
\begin{align*}
& \lim_{\delta \to 0} \, Q_1^\delta + Q_2^\delta + Q_3^\delta=\int_0^t \left ( \partial_s f(s,X(s)) +  Lf(s,X(s))  \right )\mathbf{1}_{\{ X(s) \neq \sfv \}} ds \\
& \,  + \int_0^t \sigma(X(s)) \partial_x f(s,X(s)) \mathbf{1}_{\{X(s) \neq \sfv \}} d \widetilde W(s) + \int_0^t \left (\eta \partial_s f(s,0) +  \sum_{i=1}^N \tilde \rho_i \partial_x f_i(s,0) \right ) d\ell_X(s).
\end{align*}
We observe
\begin{align*}
&\int_0^t \left ( \partial_s f(s,X(s)) +  Lf(s,X(s))  \right )\mathbf{1}_{\{ X(s) \neq \sfv \}} ds  =\\ &\int_0^t \left ( \partial_s f(s,X(s)) +  Lf(s,X(s))  \right ) ds - \int_0^t \left ( \partial_s f(s,X(s)) +  Lf(s,X(s))  \right ) \mathbf{1}_{\{X(s) = \sfv \}} ds,
\end{align*}
but, from Proposition \ref{prop:occupation}, we have
\begin{align*}
\int_0^t \left ( \partial_s f(s,X(s)) +  Lf(s,X(s))  \right ) \mathbf{1}_{\{X(s) = \sfv \}} ds = \eta \int_0^t \left(\partial_s f(s,\sfv) +  Lf(s,\sfv)\right) d\ell_X(s).
\end{align*}
Then, we conclude
\begin{align*}
& f(t,X(t)) = f(0,\mathsf{x}) + \int_0^t \left ( \partial_s f(s,X(s)) +  Lf(s,X(s)) \right )  ds  \\
& \, + \int_0^t \sigma(X(s)) \partial_x f(s,X(s)) \mathbf{1}_{\{X(s) \neq \sfv \}} d \widetilde W(s) + \int_0^t \left ( \sum_{i=1}^N \rho_i \partial_x f_i(s,0) - \eta\, L f(s,\sfv) \right ) d \ell_X(s).
\end{align*}
\end{proof}

\section{Application to differential equations}
\label{section:representation_formulas}

Let $u \in C^2(\Gamma) \cap C_0(\Gamma)$ be a solution to
\begin{equation}\label{eq:application_elliptic}
\begin{cases}
 Lu(\sfx) - \lambda u(\sfx) = f(\sfx) \quad & \textnormal{for } \sfx \neq \sfv, \\
 u_i(0) = u_j(0) \quad & \textnormal{for all } 1 \leq i, j \leq N, \\
 \sum_{i=1}^N \rho_i u_i'(0) - \eta \lambda u(\sfv) = \eta \theta,
\end{cases}
\end{equation}
for $\lambda > 0$, $f \in PC(\Gamma)$ and $\theta \in \RR$. Let also $X$ be the sticky diffusion process generated by $L$. In the case where $f \in C(\Gamma)$, it follows from \eqref{eq:application_elliptic} that $Lu \in C(\Gamma)$ with $Lu(\sfv) = f(\sfv) + \lambda u(\sfv)$. The third condition in \eqref{eq:application_elliptic} then becomes
\[
 \sum_{i=1}^N \rho_i u_i'(0) - \eta Lu(\sfv) = \eta \theta - f(\sfv).
\]
If we also assume that $f(\sfv) = \eta \theta$, we obtain that $u \in D(L)$. Then, setting
\begin{equation}\label{eq:Feynmann_Kac_stationnary}
 v(t,\sfx) := u(\sfx)e^{- \lambda t} \quad \textnormal{for } \sfx \in \Gamma, \, t \geq 0,
\end{equation}
we have $v(t,\cdot) \in D(L)$. For $(t_n)_{n \in \NN}$ such that $t_n \xrightarrow{n \to \infty} +\infty$, a straightforward application of Dynkin's formula yields
\begin{align*}
 u(\sfx) - \mathbf E_\sfx \left[u(X(t_n))e^{-\lambda t_n} \right] & = - \mathbf E_\sfx \left [\int_0^{t_n} f(X(s)) e^{- \lambda s} \, ds \right].
\end{align*}
Passing to the limit $n \to \infty$ we conclude that
\[
 u(\sfx) = - \mathbf E_\sfx \left [\int_0^{+\infty} f(X(s)) e^{- \lambda s} \, ds \right].
\]
The Itô formula \cref{tm:Ito} allows for generalization of this fact to cases where $f \notin C(\Gamma)$ or $f(\sfv) \neq \eta \theta$.
\begin{prop}
 Let $u \in C_b^2(\Gamma)$ be a solution to \eqref{eq:application_elliptic} with $\lambda > 0$, $f \in PC(\Gamma)$ bounded and $\theta \in \RR$. Then
 \[
  u(\sfx) = - \mathbf E_\sfx \left[ \int_0^{+\infty} \left ( f(X(s)) \mathbf 1_{\{X(s) \neq \sfv \}} + \theta \mathbf 1_{\{X(s) = \sfv \}} \right) e^{- \lambda s} \, ds \right ],
 \]
 where $X$ is the sticky diffusion generated by $L$ with stickiness $\eta$.
\end{prop}

\begin{proof}
 Define $v \in C_b^{1,2}(\RR_+ \times \Gamma)$ by \eqref{eq:Feynmann_Kac_stationnary} and let $t_n$ be a sequence of positive real numbers tending to $+\infty$. Then, the It\^o formula \cref{tm:Ito} implies that
 \begin{align*}
  u(\sfx) & - \mathbf{E}_\sfx \left[u(X(t_n))e^{-\lambda t_n} \right] = \mathbf{E}_{\sfx}\left[\int_0^{t_n} \left(\lambda u(X(s)) - Lu(X(s)) \right)e^{-\lambda s} \mathbf{1}_{\{X(s) \neq \sfv \}} \, ds \right] \\
  & \qquad + \mathbf{E}_{\sfx} \left[\int_0^{t_n} \left( \lambda \eta u(\sfv) - \sum_{i=1}^N \rho_i u_i'(0) \right)e^{-\lambda s} d\ell_X(s) \right] \\
  & \quad = - \mathbf{E}_\sfx \left[ \int_0^{t_n} f(X(s))e^{-\lambda s} \mathbf{1}_{\{X(s) \neq \sfv\}} ds + \int_0^{t_n} \eta \theta e^{-\lambda s} d\ell_X(s) \right] \\
  & \quad = - \mathbf{E}_\sfx \left[ \int_0^{t_n} \left(f(X(s)) \mathbf{1}_{\{X(s) \neq \sfv\}} + \theta \mathbf{1}_{\{X(s) = \sfv \}} \right)e^{-\lambda s} \, ds \right].
 \end{align*}
 The conclusion then follows from Lebesgue's convergence theorem.
\end{proof}

Similarly, we obtain a Feynman-Kac formula for parabolic equations on $\Gamma$ with dynamic junction conditions.
\begin{prop}[Feynman-Kac]
 Let $T > 0$ and $u \in C^{1,2}_b((0,T) \times \Gamma) \cap C_b([0,T] \times \Gamma)$  be a solution to
 \begin{equation}
  \begin{cases}
   -\partial_t u(t,\sfx) - L u(t,\sfx) = f(t,\sfx) \quad & \textnormal{for all } t \in (0,T), \, \sfx \neq \sfv, \\
   u_i(t,0) = u_j(t,0) \quad & \textnormal{for all } 1 \leq i, j \leq N, \, t \in (0,T), \\
   -\eta \partial_t u(t,\sfv) - \sum_{i=1}^N \rho_i \partial_x u_i(t,0) = \eta \theta(t) \quad &  \textnormal{for all } t \in (0,T), \\
   u(T, \cdot) = u_T,
  \end{cases}
 \end{equation}
 with $f \in C([0,T], PC(\Gamma))$, $u_T \in C_b(\Gamma)$ and $\theta \in C([0,T])$. Then
 \[
  u(t,\sfx) = \mathbf E \left[u_T(X(T)) + \int_t^T f(s,X(s)) \mathbf 1_{\{X(s) \neq \sfv \}} + \theta(s) \mathbf 1_{\{X(s) = \sfv \}} ds \big \vert X(t) = \sfx \right ],
 \]
 where $X$ is the sticky diffusion generated by $L$ with stickiness $\eta$.
\end{prop}
\begin{proof}
 Applying the It\^o formula \cref{tm:Ito} to $u$, we get
 \begin{align*}
  u(t,x) & = \mathbf{E} \left[ u(T,X(T)) - \int_t^T \left(\partial_t u(s,X(s)) + Lu(s,X(s)) \right) \mathbf{1}_{\{X(s) \neq \sfv \}}\, ds \big \vert X(t) = \sfx \right] \\
  & \quad - \mathbf{E} \left[\int_t^T \left(\eta \partial_t u(s,\sfv) + \sum_{i=1}^N \rho_i \partial_x u_i(s,0) \right) d\ell_X(s) \big \vert X(t) = \sfx \right ]\\
  & = \mathbf{E} \left[u(T,X(T)) + \int_t^T f(s,X(s)) \mathbf{1}_{\{X(s) \neq \sfv \}} \, ds + \int_t^T \eta \theta(s) d\ell_X(s) \big \vert X(t) = \sfx \right] \\
  & = \mathbf{E} \left[u(T,X(T)) + \int_t^T f(s,X(s)) \mathbf{1}_{\{X(s) \neq \sfv \}} + \theta(s) \mathbf{1}_{\{X(s) = \sfv\}} \, ds \big \vert X(t) = \sfx \right].
 \end{align*}
\end{proof}

\appendix

\section{A boundary value problem}
\label{section:BVP}

We use here the same notations and assumptions as those from \cref{section:statement}.

\begin{prop}\label{prop:bvp}
 For every $\delta > 0$ and $f \in \bigoplus_{i=1}^N C_b([0,\delta))$, with $f_i(0) = f_j(0)$ for all $1 \leq i,j \leq N$, there is a unique $u \in \bigoplus_{i=1}^N C^2([0,\delta])$ such that
 \begin{equation}\label{eq:bvp}
  \begin{cases}
   L_i u_i(x) = f_i(x) \quad & \textnormal{for all } x \in (0, \delta), \, 1 \leq i \leq N, \\
	\eta L u (\sfv) - \sum_{i=1}^N \rho_i {u_i}'(0) = 0, \\
	u_i(\delta) = 0 \quad & \textnormal{for all } i \in \{1,\dots, N \}, \\
	u_i(0) = u_j(0) \quad & \textnormal{for all } i,j \in \{1,\dots, N \},
  \end{cases}
 \end{equation}
 given by
 \[
  u_i(x) = - \frac{\kappa_i  \alpha_i(x)}{\alpha_i(0)} - \beta_i(x) \quad \textnormal{for all } 0 \leq x \leq \delta,
  \]
  where
  \begin{align*}
  \kappa_i & = \left (\eta f(\sfv) - \sum_{j=1}^N \frac{\rho_j(\beta_1(0) - \beta_j(0))}{ \alpha_j(0)} \right ) \left (\sum_{j=1}^N \frac{\rho_j}{\alpha_j(0)} \right )^{-1} + \beta_1(0) - \beta_i(0) \\
  \alpha_i(x) & = \int_x^\delta \exp \left (-\int_0^y \frac{2b_i(z)}{\sigma_i^2(z)} dz \right) dy \\
  \beta_i (x) & = \int_x^\delta \left ( \int_0^y \frac{2 f(z)}{\sigma_i^2(z)} \exp \left (\int_0^z \frac{2b_i(s)}{\sigma_i^2(s)} ds  \right) dz \right ) \exp \left (-\int_0^y \frac{2b_i(z)}{\sigma_i^2(z)} dz \right) dy
 \end{align*}
 for all $1 \leq i \leq N$.
 Furthermore, we have 
 \[
  \max_{1 \leq i \leq N} \max_{x \in [0,\delta]} \left \lvert u_i''(x) \right  \rvert = O(1) \quad  \textnormal{as } \delta \downarrow 0,
 \]
 where the hidden constant only depends on $f$ through $\lVert f \rVert_{\infty}$.
\end{prop}

\begin{proof}
We divide the proof into three steps.\\

 \emph{1. Proof of existence.} Using the positivity of $\sigma^2$ \eqref{eq_diffusion_positive}, we first rewrite the problem in the form
 \[
  u_i''(x) + \frac{2 b_i(x)}{\sigma_i^2(x)} u_i'(x) = \frac{2 f(x)}{\sigma_i^2(x)} \quad \textnormal{for all } x \in (0,\delta).
 \]
 This implies that 
 \begin{equation}\label{eq:bvp_derivative}
  u_i'(x) = \left ( A_i + \int_0^x \frac{2f_i(y)}{\sigma_i^2(y)} \exp \left (\int_0^y \frac{2b_i(z)}{\sigma_i^2(z)} dz  \right) dy \right ) \exp \left (-\int_0^x \frac{2b_i(y)}{\sigma_i^2(y)} dy \right) \quad \textnormal{for } x \in (0, \delta),
 \end{equation}
 were the parameters $A_i$ are constants. In particular we have $u_i'(0) = \lim_{x \downarrow 0} u_i'(x) = A_i$. It follows that 
 \[
  u_i(x) = - \int_x^\delta \left ( A_i + \int_0^y \frac{2 f_i(z)}{\sigma_i^2(z)} \exp \left (\int_0^z \frac{2b_i(s)}{\sigma_i^2(s)} ds  \right) dz \right ) \exp \left (-\int_0^y \frac{2b_i(z)}{\sigma_i^2(z)} dz \right) dy \quad \textnormal{for } x \in (0, \delta).
 \]
 Notice in particular that $u_i(\delta) = 0$. Setting 
 \begin{align*}
  \alpha_i(x) & = \int_x^\delta \exp \left (-\int_0^y \frac{2b_i(z)}{\sigma_i^2(z)} dz \right) dy \\
  \beta_i (x) & = \int_x^\delta \left ( \int_0^y \frac{2 f_i(z)}{\sigma_i^2(z)} \exp \left (\int_0^z \frac{2b_i(s)}{\sigma_i^2(s)} ds  \right) dz \right ) \exp \left (-\int_0^y \frac{2b_i(z)}{\sigma_i^2(z)} dz \right) dy
 \end{align*}
 we get $u_i(x) = - A_i \alpha_i(x) - \beta_i(x)$. The continuity of $u$ then implies $A_i \alpha_i(0) + \beta_i(0) = A_j \alpha_j(0) + \beta_j(0)$ for all $1 \leq i,j \leq N$. Since $\alpha_i(0) > 0$ we obtain
 \[
  A_i = \frac{A_1 \alpha_1(0) + \beta_1(0) - \beta_i(0)}{\alpha_i(0)} \quad \textnormal{for all } 1 \leq i \leq N.
 \]
The second condition in \eqref{eq:bvp} then reads 
\begin{align*}
 \eta f(\sfv) &= \eta L u(\sfv) = \sum_{i=1}^N \rho_i u_i'(0) = \sum_{i=1}^N \rho_i A_i = A_1 \alpha_1(0) \left (\sum_{i=1}^N \frac{\rho_i}{\alpha_i(0)} \right ) + \sum_{i=1}^N \frac{\rho_i (\beta_1(0) - \beta_i(0))}{\alpha_i(0)}
\end{align*}
so that 
\[
 A_1 = \left (\eta f(\sfv) - \sum_{i=1}^N \frac{\rho_i(\beta_1(0) - \beta_i(0))}{ \alpha_i(0)} \right ) \left (\sum_{i=1}^N \frac{\rho_i \alpha_1(0)}{\alpha_i(0)} \right )^{-1}.
\]
\\

\emph{2. Proof of uniqueness. } Clearly it is enough to prove that $v=0$ is the unique element in $\bigoplus_{i=1}^N C_b([0,\delta))$ satisfying
	\begin{equation}
		\begin{cases}
		 \frac{1}{2} \sigma^2_i(x) {v_i}''(x) + b_i(x){v_i}'(x) = 0\quad & \textnormal{for all } x \in (0, \delta), \, 1 \leq i \leq N, \\
		 \sum_{i=1}^N \rho_i {v_i}'(0) = 0, \\
		 v_i(0) = v_j(0) \quad & \textnormal{for all } i,j \in \{1,\dots, N \}, \\
		  v_i(\delta) = 0 \quad & \textnormal{for all } i \in \{1,\dots, N \}.
		\end{cases}
	\end{equation}
Applying the strong maximum principle \cite[Theorem 3 p.349]{E2010} to $v_i$ we deduce that it cannot achieve its maximum (resp. minimum) on $(0,\delta)$ unless it is constant. In the case where $v_i$ is constant we deduce that $v_i = 0$ from the Dirichlet boundary condition. We therefore proceed by assuming that the $v_i$ is not all constant. If the maximum (resp. minimum) of $v$ is attained at $0$, then, Hopf's lemma \cite[Lemma p.347]{E2010}, or the simpler result \cite[Theorem 2 p.4]{PW1984}, implies that $v_i'(0) < 0$ (resp. $v_i'(0) > 0$) for all $1 \leq i \leq N$, which contradicts the Kirchhoff conditions. We conclude that the maximum (resp. minimum) must be $v_i(\delta) = 0$ for some $1 \leq i \leq N$, so that $v \leq 0$ (resp. $v \geq 0$). This proves uniqueness. \\

\emph{3. Proof of estimate.} From \eqref{eq:bvp_derivative} and uniqueness of solutions we obtain 
\begin{align*}
 u_i''(x) & =  \left (\frac{2 f(x)}{\sigma_i^2(x)} - \frac{2b_i(x)}{\sigma_i^2(x)} \left ( \int_0^x \frac{2 f_i(y)}{\sigma_i^2(y)} \exp \left (\int_0^y \frac{2b_i(z)}{\sigma_i^2(z)} dz  \right) dy \right ) \right )\exp \left (-\int_0^x \frac{2b_i(y)}{\sigma_i^2(y)} dy \right) \\
 & \quad - \frac{2A_i b_i(x)}{\sigma_i^2(x)}\exp \left (-\int_0^x \frac{2b_i(y)}{\sigma_i^2(y)} dy \right)
\end{align*}
so that 
\[
 \lvert u_i''(x) \rvert \leq \frac{2(\lVert f \rVert_{\infty} + A_i \lVert b \rVert_\infty)}{\sigma_0^2} + \frac{4 \lVert b \rVert_\infty \lVert f \rVert_{\infty} \delta }{\sigma_0^4} \exp \left ({\frac{2\delta \lVert b \rVert_\infty}{\sigma_0^2}} \right ).
\]
Notice also that 
\[
 \delta \exp \left (- \frac{2\delta \lVert b \rVert_\infty}{\sigma_0^2} \right ) \leq \alpha_i(0) \leq \delta
\]
and 
\[
 0 < \beta_i(0) \leq \frac{2 \delta^2}{\sigma_0^2} \exp \left (\frac{2 \delta \lVert b \rVert_\infty}{\sigma_0^2} \right ).
\]
So that $A_i = O(1)$ as $\delta \downarrow 0$, with hidden constant depending on $f$ through $\lVert f \rVert_{\infty}$. The conclusion follows.
\end{proof}

\section*{Acknowledgment}
\noindent
The authors sincerely express their gratitude to Fabio Camilli and Mirko D'Ovidio for their ideas.

J.B. was partially supported by the ANR (Agence Nationale de la Recherche) through the
COSS project ANR-22-CE40-0010 and the Centre Henri Lebesgue ANR-11-LABX-0020-01 and by Rennes Métropole through the Collège doctoral de Bretagne. This work was initiated while J.B. was visiting  Sapienza Università di Roma. He
wishes to thank the university for its hospitality.

F.C. thanks Sapienza and the group INdAM-GNAMPA for the support.
The research of F.C. has been mostly funded by MUR under the project PRIN 2022 - 2022XZSAFN: Anomalous Phenomena on Regular and Irregular Domains:
Approximating Complexity for the Applied Sciences - CUP B53D23009540006 - PNRR M4.C2.1.1.\\
Web Site: \url{https://www.sbai.uniroma1.it/~mirko.dovidio/prinSite/index.html}.

\bibliographystyle{abbrvnat}
\bibliography{sticky-diffusion-graphs.bib}
\end{document}